\documentclass{amsart}
\usepackage[english]{babel}
\usepackage{amssymb,enumerate,amsmath}
\numberwithin{equation}{section}
\newtheorem{theorem}{Theorem}[section]
\newtheorem{thmA}{Theorem}

\newtheorem{corollary}[theorem]{Corollary}
\newtheorem{lemma}[theorem]{Lemma}

\newtheorem{remark}{Remark}

\newtheorem{conjecture}{Conjecture}

\newenvironment{enumerateroman}{\begin{enumerate}[\textup{(}i\textup{)}] }{\end{enumerate}}

\begin{document}

\title[Mean curvature flow and differentiable sphere theorem]
{Mean curvature flow of arbitrary codimension in spheres and sharp
differentiable sphere theorem}

\author{Li Lei and Hongwei Xu}
\address{Center of Mathematical Sciences \\ Zhejiang University \\ Hangzhou 310027 \\ China\\}
\email{lei-li@zju.edu.cn; xuhw@zju.edu.cn}
\date{}
\thanks{Research supported by the National Natural Science Foundation of China, Grant Nos. 12071424, 11371315}

\begin{abstract}
 In this paper, we investigate Liu-Xu-Ye-Zhao's conjecture {\cite{liu2011extension}} and prove a sharp convergence theorem for the mean curvature flow of arbitrary codimension in spheres which improves
 the convergence theorem of Baker {\cite{baker2011mean}} as well as the differentiable sphere theorems of Gu-Xu-Zhao
{\cite{MR3005061,XuGu2010,MR2550209}}.
\end{abstract}

{\maketitle}

\section{Introduction}
It plays an important role in differential geometry to study
pinching theory in global differential geometry and geometric
analysis. Motivated by geometrical and topological pinching results
on submanifolds, we will investigate the convergence theorem for the
mean curvature flow and the differentiable sphere theorem for
submanifolds under a sharp pinching condition. Let
$\mathbb{F}^{n+q}(c)$ be an $(n+q)$-dimensional simply connected
space form of constant curvature $c$, and $M^{n}$ an
$n(\geq2)$-dimensional submanifold in $\mathbb{F}^{n+q}(c)$. Denote
by $H$ and $h$ the mean curvature vector and the second fundamental
form of $M$, respectively. Assume that $|H|^{2}+4( n-1 ) c\ge0$. We
define
\begin{equation}
  \alpha (n,|H|,c ) =n c+ \frac{n}{2 ( n-1 )} |H|^{2}- \frac{n-2}{2 ( n-1 )} \sqrt{|H|^{4}
  +4 ( n-1 ) c |H|^{2}} ,
\end{equation}
and $\alpha_{1} (n,|H|)= \alpha (n,|H|,1)$.

Since 1973, Okumura \cite{MR0317246,MR0353216}, Yau \cite{Yau} and
many other authors tried to generalize the famous
Simons-Lawson-Chern-do Carmo-Kobayashi rigidity theorem
{\cite{MR0273546,MR0238229,MR0233295}} to the case where $M$ is a
closed submanifold with parallel mean curvature in a sphere, and got
partial results. In \cite{Xu1}, Xu proved the generalized
Simons-Lawson-Chern-do Carmo-Kobayashi theorem for closed
submanifolds with parallel mean curvature in a sphere. The following
refined version of the generalized Simons-Lawson-Chern-do
Carmo-Kobayashi theorem was obtained by Li-Li {\cite{MR1161925}} for
$H=0$ and by Xu {\cite{MR1241055}} for $H\neq0$.
\begin{thmA}\label{rigidity}
Let $M$ be an $n$-dimensional oriented compact submanifold with
parallel mean curvature in the unit sphere $\mathbb{S}^{n+q}$. If $|
h |^{2} \leq C(n,q,|H|),$ then $M$ is either congruent to a round
sphere, a Clifford hypersurface in an $(n+1)$-sphere, or the
Veronese surface in a $4$-sphere. Here $C(n,q,|H|)$ is defined by
\begin{equation*}  C(n,q,|H|) = \left\{\begin{array}{ll}
     \alpha_{1} (n, | H | ) , & q=1,  \operatorname{or}  q=2
      \operatorname{and}  | H | \neq 0,\\
     \min \left\{ \alpha_{1} (n, | H | ) , \frac{2n}{3} + \frac{5}{3n} | H |^{2}
     \right\} , &  \operatorname{otherwise}.
   \end{array}\right. \end{equation*}
   \end{thmA}

In {\cite{MR0324529}}, Lawson and Simons proved that if $M^{n}$ ($n
\ge 5$) is an oriented compact submanifold in $\mathbb{S}^{n+q}$
satisfying $| h |^{2} <2 \sqrt{n-1}$, then $M$ is homeomorphic to
$\mathbb{S}^{n}$. More generally, Shiohama and Xu {\cite{MR1458750}}
proved the optimal topological sphere theorem.

\begin{thmA}Let $M^{n}$ ($n
\ge 4$) be an oriented complete submanifold in $\mathbb{F}^{n+q} ( c
)$ with $c\ge0$. If $\sup_{M} ( | h |^{2} - \alpha ( n,|H|,c ) )
<0$, then $M$ is homeomorphic to $\mathbb{S}^{n}$.\end{thmA}

In \cite{MR2550209}, the second author and Zhao initiated the study
of differentiable pinching problem on submanifolds of arbitrary
codimension. Making use of the convergence results of Hamilton and
Brendle for Ricci flow and the Lawson-Simons formula for the
nonexistence of stable currents, Xu and Zhao \cite{MR2550209} proved
the following.

\begin{thmA} Let $M$ be an $n$-dimensional ($n \ge 4$)
oriented complete submanifold in the unit sphere $\mathbb{S}^{n+q}$.
Then\\
\hspace*{2mm}$(i)$ if $n=4,5,6$ and $\sup_M(| h |^2-\alpha_1(n,|H|))<0$, then $M$ is diffeomorphic to $\mathbb{S}^{n}$;\\
\hspace*{2mm}$(ii)$ if $n\geq7$ and $|h|^2<2\sqrt{2}$, then $M$ is
diffeomorphic to  $\mathbb{S}^{n}$.\end{thmA}

In \cite{MR3005061,XuGu2010}, Gu and Xu proved a differentiable
sphere theorem for submanifolds in a Riemannian manifold via the
Ricci flow. Consequently, they got the following.
\begin{thmA}
Let $M$ be an $n$-dimensional oriented compact submanifold in
the $(n+q)$-dimensional space form $\mathbb{F}^{n+q}
( c )$  with $c\ge0$ and $|H|^2+n^2c>0$. Then\\
\hspace*{2mm}$(i)$ if $n=2$ and $| h |^{2}\leq
2c+|H|^2$, then $M$ is diffeomorphic to $\mathbb{S}^{2}$, or $M$ is flat;\\
\hspace*{2mm}$(ii)$ if $n=3$ and $|h|^2<2c+\frac{|H|^2}{2}$, then
$M$ is
diffeomorphic to $\mathbb{S}^{3}$;\\
\hspace*{2mm}$(iii)$ if $n\ge4$ and $| h |^{2}\leq
2c+\frac{|H|^2}{n-1}$, then $M$ is diffeomorphic to
$\mathbb{S}^{n}$.
\end{thmA}
\begin{remark} When $n\geq4$, $c=0$ and $|H|>0$, Andrews-Baker
\cite{MR2739807} independently proved the same differentiable sphere
theorem via the mean curvature flow of submanifolds of high
codimension. Afterwards, Baker {\cite{baker2011mean}} gave another
proof of the differentiable sphere theorem for $n\geq4$ and $c>0$.
Later, Liu-Xu-Ye-Zhao {\cite{MR3078951}} extended the differentiable
sphere theorem above to the case where $n\geq4$, $c<0$ and
$|H|^2+n^2c>0$.
\end{remark}

We refer the readers to
\cite{MR0343217,ChKa2001,MR1633163,XG2} for
further discussions on rigidity theorems of submanifolds with
parallel mean curvature, to
\cite{MR3005061,LiWang2014,liu2011extension,liu2012mean,liu2013mean,Shiohama,XG2}
for more discussions on sphere theorems of submanifolds and to
\cite{Berger,BW,Brendle1,Brendle2,MR2449060,BS2,CD,CTZ,GP,GS,Hamilton,Huisken2,MM,OSY,Perelman,Petersen,Shiohama}
for various sphere theorems of Riemannian manifolds.

Let $F_{0} :M^{n} \rightarrow \mathbb{F}^{n+q}(c)$ be an
$n$-dimensional submanifold immersed in the space form
$\mathbb{F}^{n+q}(c)$. The mean curvature flow with initial value
$F_{0}$ is a smooth family of immersions $F: M \times [ 0,T )
\rightarrow \mathbb{F}^{n+q}(c)$ satisfying
\begin{equation}
  \left\{\begin{array}{l}
    \frac{\partial}{\partial t} F ( x,t ) =H ( x,t ) ,\\
    F ( \cdot ,0 ) =F_{0} ,
  \end{array}\right.
\end{equation}
where $H ( x,t )$ is the mean curvature vector of $M_{t} =F_{t} ( M
)$, $F_{t} =F ( \cdot ,t )$.

In 1980's, Huisken, the founder of the mean curvature flow theory,
initiated the study of the mean curvature flow for compact
hypersurfaces {\cite{MR772132,MR837523,MR892052}}. In 1987, Huisken
{\cite{MR892052}} verified the convergence theorem for the mean
curvature flow of compact hypersurfaces in the spherical space form
of constant curvature $c$ under the pinching condition $| h |^{2} <
\frac{1}{n-1} | H |^{2}+2c$.

For higher codimensional cases, the convergence theorems for the
mean curvature flow in Euclidean spaces, spheres and hyperbolic
spaces were proved by Andrews-Baker {\cite{MR2739807}}, Baker
{\cite{baker2011mean}} and Liu-Xu-Ye-Zhao {\cite{MR3078951}},
respectively. The unified version of the convergence theorems due to
Andrews, Baker, Liu, Xu, Ye and Zhao
{\cite{MR2739807,baker2011mean,MR3078951}} can be summarized as
follows.

\begin{thmA}Let $F_{0} :M^{n} \rightarrow \mathbb{F}^{n+q}(c)$ be an $n$-dimensional ($n \ge 4$)
closed submanifold in the space form with constant curvature $c$. If
$F_{0}$ satisfies $| h |^{2} \le \frac{1}{n-1} | H |^{2} +2c$,
where $|H|^2+n^2c>0$, then the mean curvature flow with initial
value $F_{0}$ converges to a round point in finite time, or $c>0$
and $F_{t}$ converges to a totally geodesic sphere as $t \rightarrow
\infty$.
\end{thmA}

Notice that the pinching condition $| h |^{2} < \frac{1}{n-1} | H
|^{2}+2c$ implies that the sectional curvature of $M$ is positive.
On the other hand, the pinching condition $| h |^{2}<\alpha ( n, | H
| ,c )$ for $|H|^2+n^2c>0$ and $c\neq 0$ implies that the Ricci
curvature of the initial submanifold is positive, but does not imply
positivity of the sectional curvature. Hence we need to investigate
the convergence problem for the mean curvature flow in space forms
under the pinching condition $|h|^{2}<\alpha(n,|H|,c)$, for
$|H|^2+n^2c>0$ and $c\neq 0$. Recently, the authors
{\cite{lei2015optimal}} proved an optimal convergence theorem for
the mean curvature flow of arbitrary codimension in hyperbolic
spaces.

Motivated by the rigidity and sphere theorems for submanifolds in
spheres, Liu-Xu-Ye-Zhao {\cite{liu2011extension}} proposed the
following.

\begin{conjecture}
  Let $M_{0}$ be an n-dimensional complete submanifold in the sphere
  $\mathbb{S}^{n+q} \left( 1/ \sqrt{c} \right)$. Suppose that
  $\sup_{M_{0}} ( | h |^{2} - \alpha (n, | H |,c ) ) <0$. Then the mean
  curvature flow with initial value $M_{0}$ converges to a round point in
  finite time, or converges to a totally geodesic sphere as $t \rightarrow
  \infty$. In particular, $M_{0}$ is diffeomorphic to $\mathbb{S}^{n}$.
\end{conjecture}

In particular, noting that $\min_{|H|}   \alpha_1(n,|H|) =2
\sqrt{n-1}$, we have the following.

\begin{conjecture}
  Let $F_{0} :M^{n} \rightarrow \mathbb{S}^{n+q}$ be an $n$-dimensional
  closed submanifold satisfying $| h |^{2} <2 \sqrt{n-1}$. Then
  the mean curvature flow with initial value $F_{0}$ converges to a
  round point or a totally geodesic sphere. In particular, $M_0$ is diffeomorphic to $\mathbb{S}^{n}$.
\end{conjecture}

Conjecture 2 in dimension three implies the famous Lawson-Simons
conjecture {\cite{MR0324529}}, which states that if $M$ is a
$3$-dimensional compact submanifold in $\mathbb{S}^{3+q}$, and if $|
h |^{2} <2 \sqrt{2}$, then $M$ is diffeomorphic to $\mathbb{S}^{3}$.
Up to now, the Lawson-Simons conjecture is still open. After the
work on the mean curvature flow of hypersurfaces in spheres due to
Li-Xu-Zhao {\cite{lxz2013optimal}}, the authors
{\cite{lei2014sharp}} obtained a refined version of Huisken's
convergence theorem for the mean curvature flow of hypersurfaces
{\cite{MR892052}}. For more convergence results on the mean
curvature flow with applications in sphere theorems, we refer the
readers to
{\cite{liu2011extension,liu2012mean,liu2013mean,PiSi2015}}.

The purpose of the present paper is to investigate Liu-Xu-Ye-Zhao's
conjectures and prove the following sharp convergence theorem for
the mean curvature flow of arbitrary codimension in spheres.

\begin{theorem}
  \label{theo1}Let $F_{0} :M^{n} \rightarrow \mathbb{S}^{n+q} \left( 1/
  \sqrt{c} \right)$ be an n-dimensional ($n \ge 6$) closed submanifold
  in a sphere. If $F_{0}$ satisfies
  \begin{equation*}  | h |^{2} < \gamma( n,|H|,c ), \end{equation*}
  then the mean curvature flow with initial value $F_{0}$ has a unique smooth
  solution $F: M \times [ 0,T ) \rightarrow \mathbb{S}^{n+q} \left( 1/
  \sqrt{c} \right)$, and $F_{t}$ converges to a round point in finite time or
  converges to a totally geodesic sphere as $t \rightarrow \infty$. Here $\gamma( n,|H|,c )$ is an explicit
  positive scalar defined by
$ \gamma( n,|H|,c ):=\min \{ \alpha ( | H |^{2} ) , \beta ( | H
|^{2} ) \} $,  where
\begin{equation}
 \alpha ( x ) =n c+ \frac{n}{2 ( n-1 )} x- \frac{n-2}{2 ( n-1 )} \sqrt{x^{2}
  +4 ( n-1 ) c x},
\end{equation}
\begin{equation}
  \beta ( x ) = \alpha ( x_{0} ) + \alpha' ( x_{0} ) ( x-x_{0} ) + \frac{1}{2}
  \alpha'' ( x_{0} ) ( x-x_{0} )^{2} , \label{beta}
\end{equation}
\begin{equation*}  x_{0} = \frac{2  n+2 }{n-4} \sqrt{n-1} \Big( \sqrt{n-1} - \frac{n-4}{2
   n+2 } \Big)^{2} c. \end{equation*}
\end{theorem}

\begin{remark} Since $\gamma ( n,|H|,c )=\alpha(n,|H|,c)$, for $|H|^2\ge x_{0}$, Theorem \ref{theo1} is sharp. A computation shows that (i)
$\gamma(n,|H|,c) > \frac{1}{n-1}|H|^{2} +2c$, for $n \ge 6$;
(ii) $\gamma( n,|H|,c )
> \frac{3n + 3}{2 n + 6} \sqrt{n-1} c$, for $n
\ge 6$; (iii) $ \gamma(n,|H|,c)> 2 \sqrt{2}c$, for $n
\ge 7$.  Therefore, Theorem \ref{theo1} substantially improves
Theorems C, D and E. The pinching condition in Theorem \ref{theo1}
implies that the Ricci curvature of the initial submanifold is
positive, but does not imply positivity of the sectional
curvature.\end{remark}

Since $\gamma( n,|H|,c )> \frac{3n + 3}{2 n + 6} \sqrt{n-1} c$, we obtain

\begin{theorem}
  Let $F_{0} :M^{n} \rightarrow \mathbb{S}^{n+q} \left( 1/
  \sqrt{c} \right)$ be an $n$-dimensional
  ($n \ge 6$) closed submanifold. If $F_{0}$
  satisfies \[   | h |^{2} < \frac{3n + 3}{2 n + 6} \sqrt{n-1}\, c, \]
  then the mean curvature flow with
  initial value $F_{0}$ has a unique smooth solution $F: M \times [ 0,T )
  \rightarrow \mathbb{S}^{n+q} \left( 1/
  \sqrt{c} \right)$, and $F_{t}$ converges to a round point
  in finite time or converges to a totally geodesic sphere as $t \rightarrow
  \infty$.
\end{theorem}

As a consequence of Theorem \ref{theo1}, we obtain a differentiable sphere theorem.

\begin{theorem}
  \label{theo5}
  Let $M_0$  be
  an n-dimensional ($n \ge 6$) closed submanifold in $ \mathbb{S}^{n+q} \left( 1/ \sqrt{c} \right)$.
  If $M_{0}$ satisfies $| h |^{2} < \gamma ( n,|H|,c )$, then $M_0$ is diffeomorphic
  to the standard $n$-sphere $\mathbb{S}^{n}$. In particular, if $M_{0}$
  satisfies
  $| h |^{2} < \frac{3n + 3}{2 n + 6} \sqrt{n-1}c, $
  then $M_0$ is diffeomorphic
  to $\mathbb{S}^{n}$.
\end{theorem}

Under the weakly pinching condition, we get the following
convergence theorem.

\begin{theorem}
  \label{theo2}Let $F_{0} :M^{n} \rightarrow \mathbb{S}^{n+q} \left( 1/
  \sqrt{c} \right)$ be an n-dimensional ($n \ge 6$) closed submanifold
  in a sphere. If $F_{0}$ satisfies
  $ | h |^{2} \le \gamma( n,|H|,c )
     ,$
  then the mean curvature flow with initial value $F_{0}$ has a unique smooth
  solution $F: M \times [ 0,T ) \rightarrow \mathbb{S}^{n+q} \left( 1/
  \sqrt{c} \right)$, and either
  \begin{enumerateroman}
    \item $T$ is finite, and $F_{t}$ converges to a round point as $t
    \rightarrow T$,

    \item $T= \infty$, and $F_{t}$ converges to a totally geodesic sphere as
    $t \rightarrow \infty$, or

    \item $T$ is finite,  $M_{t}$ is congruent to $\mathbb{S}^{n-1} (
    r_{1} ) \times \mathbb{S}^{1} ( r_{2} )$, where $r_{1}^{2} +r_{2}^{2} =1/c$, $r_{1}^{2} = \frac{n-1}{n c} ( 1-
    \mathrm{e}^{2n c ( t-T )} )$, and $F_{t}$ converges to a great circle as
    $t \rightarrow T$.
  \end{enumerateroman}
\end{theorem}

As a consequence of Theorem \ref{theo2}, we have the following
classification theorem.

\begin{corollary}
  Let $M_0$ be an n-dimensional ($n \ge 6$) closed submanifold in
  $\mathbb{S}^{n+q} \left( 1/ \sqrt{c} \right)$ which satisfies $| h |^{2}
  \le \gamma( n,|H|,c )$. Then $M_0$
  is either diffeomorphic to the standard $n$-sphere $\mathbb{S}^{n}$, or
  congruent to $\mathbb{S}^{n-1} ( r_{1} ) \times \mathbb{S}^{1} ( r_{2} )$,
  where $r_{1}^{2} +r_{2}^{2} =1/c$ and $r_{1}^{2} < \frac{n-1}{n c}$.
\end{corollary}

Motivated by the convergence theorems of Ricci flow due to Hamilton
\cite{Hamilton} and Chen-Tang-Zhu \cite{CTZ}, we propose the
following problem: Is it possible for one to prove a sharp
convergence theorem of mean curvature flow in low dimensions?

\section{Notations and formulas}

Let $( M^{n} ,g )$ be a Riemannian manifold isometrically immersed in
$\mathbb{S}^{n+q} \left( 1/ \sqrt{c} \right)$. Let $T M$ and $N M$ be the
tangent bundle and normal bundle of $M$, respectively. For sections of the
bundle $T M \oplus N M$, we denote by $( \cdot )^{\top}$ and $( \cdot
)^{\bot}$ the projections onto $T M$ and $N M$, respectively.

Denote by $\overline{\nabla}$ the Levi-Civita connection of the ambient
space. We use the same symbol $\nabla$ to represent the connections
of $T M$ and $N M$. Let $\Gamma ( T M )$ and $\Gamma ( N M )$ be the
spaces of smooth sections of the bundles. The connections of $T M$
and $N M$ are given by $\nabla_{X} Y= ( \overline{\nabla}_{X} Y )^{\top}$
and $\nabla_{X} \xi = ( \overline{\nabla}_{X} \xi )^{\bot}$, for $X,Y \in
\Gamma ( T M )$, $\xi \in \Gamma ( N M )$. The second fundamental
form of $M$ is defined as $h ( X,Y ) = ( \overline{\nabla}_{X} Y
)^{\bot}$.

We shall make use of the following convention on the range of indices:
\begin{equation*} 1 \leq i,j,k,\cdots \leq n,\hspace{2em} 1 \leq \alpha,\beta,\gamma,\cdots \leq q.\end{equation*}
Let $\{ e_{i} \}$ be a local orthonormal frame for the tangent
bundle, and $\{ \nu_{\alpha}
 \}$ a local orthonormal frame for the normal bundle.
Then the mean curvature vector is defined as $H= \sum_{i} h(e_i,e_i)$.
With the local frame, the components of $h$ and $H$ are given by
$h^{\alpha}_{i j}=\langle h(e_i,e_j),\nu_\alpha\rangle$, $H^\alpha=\sum_{i} h^{\alpha}_{i i}$.
Let $\mathring{h} =h- \tfrac{1}{n} g \otimes H$ be the traceless second
fundamental form. Its norm satisfies $|\mathring{h}|^{2} = | h |^{2} -
\frac{1}{n} | H |^{2}$.

We have the following estimates for $| \nabla h |$ and $| \nabla H |$.

\begin{lemma}
  \label{dA2}For any submanifold in a sphere, we have
  \begin{enumerateroman}
    \item $| \nabla h |^{2} \geq \frac{3}{n+2} | \nabla H |^{2}$,

    \item $\big| \nabla | H |^{2} \big| \leq 2 | H |   | \nabla H |$.
  \end{enumerateroman}
\end{lemma}

The proof of (i) is the same as in {\cite{MR2739807,MR772132}}, and (ii)
follows from the Cauchy-Schwarz inequality.

As in {\cite{MR2739807,baker2011mean}}, we define the following scalars on
$M$.
\begin{equation*}  R_{1} = \sum_{\alpha , \beta} \Big( \sum_{i,j} h^{\alpha}_{i j}
   h^{\beta}_{i j} \Big)^{2} + \sum_{i,j, \alpha , \beta} \Big( \sum_{k} (
   h^{\alpha}_{i k}  h^{\beta}_{j k} -h^{\beta}_{i k}  h^{\alpha}_{j k} )
   \Big)^{2} , \end{equation*}
\begin{equation*}  R_{2} = \sum_{i,j} \Big( \sum_{\alpha} H^{\alpha} h^{\alpha}_{i j}
   \Big)^{2} , \hspace{2em} R_{3} = \sum_{i,j,k, \alpha , \beta} H^{\alpha}
   h_{i k}^{\alpha} h_{i j}^{\beta} h_{j k}^{\beta} . \end{equation*}
We have the following identity for the Laplacian of $|\mathring{h}|^{2}$.

\begin{equation}
  \frac{1}{2} \Delta |\mathring{h}|^{2} = \left\langle \mathring{h} , \nabla^{2} H
  \right\rangle + | \nabla h |^{2} - \frac{1}{n} | \nabla H |^{2} + n c
  |\mathring{h}|^{2} -R_{1} +R_{3} . \label{lapho2}
\end{equation}

At a fixed point in $M$, we choose an orthonormal frame $\{ \nu_{\alpha} \}$
for the normal space and an orthonormal frame $\{ e_{i} \}$ for the tangent
space, such that $H= | H | \nu_{1}$ and $( h^{1}_{i j} )$ is diagonal. Let
$\mathring{\lambda}_{i}$ be the diagonal elements of $(
\mathring{h}^{1}_{i j} )$. Thus we have $\mathring{\lambda}_{i}
=h^{1}_{i i} - \frac{1}{n} | H |$ and $\mathring{h}^{\alpha}_{i j}
=h^{\alpha}_{i j}$ for $\alpha >1$. We split $|\mathring{h}|^{2}$ into three parts
$$|\mathring{h}|^{2} =P_{1} +P_{2}, \quad P_{2} =Q_{1} +Q_{2},$$ where
\begin{equation*}  P_{1} = \sum_{i} \mathring{\lambda}_{i}^{2} , \hspace{1em} Q_{1} =
   \underset{i}{\underset{\alpha >1}
{\sum}} \left( h^{\alpha}_{i i} \right)^{2} , \hspace{1em}
   Q_{2} = \underset{i \neq j}{\underset{\alpha >1}
{\sum}} \left( h^{\alpha}_{i j} \right)^{2} . \end{equation*}
With the special frame, we have
\begin{eqnarray}
  R_{1} & = & P_{1}^{2} + \frac{2}{n} P_{1} | H |^{2} + \frac{1}{n^{2}} | H
  |^{4} \nonumber\\
  &  & +2 \sum_{\alpha >1} \Big( \sum_{i} \mathring{\lambda}_{i}
  h^{\alpha}_{i i} \Big)^{2} + \sum_{\alpha , \beta >1} \Big(
  \sum_{i,j} h^{\alpha}_{i j}   h^{\beta}_{i j}
  \Big)^{2} \\
  &  & +2 \underset{i \neq j}{\underset{\alpha >1}
{\sum}} \left( \left( \mathring{\lambda}_{i} - \mathring{\lambda}_{j}
  \right) h^{\alpha}_{i j} \right)^{2} +
  \underset{i,j}{\underset{\alpha , \beta >1}
{\sum}} \Big( \sum_{k} \left( h^{\alpha}_{i k}
  h^{\beta}_{j k} - h^{\alpha}_{j k}
  h^{\beta}_{i k} \right) \Big)^{2} , \nonumber
\end{eqnarray}
\begin{equation}
  R_{2} = \sum_{i,j} ( | H | h^{1}_{i j} )^{2} = | H |^{2} \Big( P_{1} +
  \frac{1}{n} | H |^{2} \Big) , \label{R2equa}
\end{equation}
and
\begin{equation}
  R_{3} = \frac{| H |^{4}}{n^{2}}  + \frac{| H |^{2}}{n} (3P_{1} +P_{2}) + | H
  | \sum_{\alpha ,i} \mathring{\lambda}_{i} \left( \mathring{h}^{\alpha}_{i i}
  \right)^{2} + \frac{| H |}{2} \underset{i \neq j}{\underset{\alpha >1}
{\sum}} \left( \mathring{\lambda}_{i} + \mathring{\lambda}_{j} \right)
  \left( h^{\alpha}_{i j} \right)^{2} .
\end{equation}
Using the Cauchy-Schwarz inequality we get $$\sum_{\alpha >1} \Big( \sum_{i}
\mathring{\lambda}_{i}   h^{\alpha}_{i i} \Big)^{2} \leq
P_{1} Q_{1}.$$ We also have
\begin{equation*}  \underset{i \neq j}{\underset{\alpha >1}
{\sum}} \left( \left( \mathring{\lambda}_{i} - \mathring{\lambda}_{j}
   \right) h^{\alpha}_{i j} \right)^{2} \leq
   \underset{i \neq j}{\underset{\alpha >1}
{\sum}} 2 \left( \mathring{\lambda}_{i}^{2} +
   \mathring{\lambda}_{j}^{2} \right) \left( h^{\alpha}_{i j}
   \right)^{2} \leq 2P_{1} Q_{2} . \end{equation*}
It follows from Theorem 1 of {\cite{MR1161925}} that
\begin{equation*}  \sum_{\alpha , \beta >1} \Big( \sum_{i,j} h^{\alpha}_{i j}
   h^{\beta}_{i j} \Big)^{2} + \underset{i,j}{\underset{\alpha , \beta >1}
{\sum}} \Big( \sum_{k} \left( h^{\alpha}_{i k}
   h^{\beta}_{j k} - h^{\alpha}_{j k}
   h^{\beta}_{i k} \right) \Big)^{2} \leq \frac{3}{2}
   P_{2}^{2} . \end{equation*}
Thus we obtain
\begin{eqnarray}
  R_{1} & \leq & P_{1}^{2} + \frac{2}{n} P_{1} | H |^{2} +
  \frac{1}{n^{2}} | H |^{4} +2P_{1} Q_{1} +4P_{1} Q_{2} + \frac{3}{2}
  P_{2}^{2} \nonumber\\
  & \leq & |\mathring{h}|^{4} + \frac{2}{n} P_{1} | H |^{2} + \frac{1}{n^{2}} |
  H |^{4} +2P_{2} |\mathring{h}|^{2} - \frac{3}{2} P_{2}^{2} .  \label{R1less}
\end{eqnarray}
From (\ref{R2equa}) and (\ref{R1less}), we have
\begin{lemma}
  \label{R12} $R_1$ and $R_2$ satisfy
  \begin{enumerateroman}
  \item  $R_{1} - \frac{1}{n} R_{2} \leq | \mathring{h} |^{4} + \frac{1}{n}
  | \mathring{h} |^{2} | H |^{2} +2P_{2} | \mathring{h} |^{2} - \frac{3}{2} P_{2}^{2} -
  \frac{1}{n} P_{2} | H |^{2}$,

  \item  $R_{2} = | \mathring{h} |^{2} | H |^{2} + \frac{1}{n} | H |^{4} -P_{2} | H
  |^{2}$.
  \end{enumerateroman}
\end{lemma}

By an algebraic inequality (see {\cite{MR1289187,MR1633163}}),
we have
\begin{eqnarray*}
   \sum_{\alpha, i} \mathring{\lambda}_i ( \mathring{h}^{\alpha}_{i
  i} )^2 & \geq & -\tfrac{n - 2}{\sqrt{n (n - 1)}} \Big( \sum_i
  \mathring{\lambda}_i^2 \Big)^{\frac{1}{2}} \Big( \sum_{\alpha, i} (
  \mathring{h}^{\alpha}_{i i} )^2 \Big)\\
  &=& - \tfrac{n - 2}{\sqrt{n (n - 1)}}
  \sqrt{P_1} ( | \mathring{h} |^2 - Q_2 )\\
  & \geq & - \tfrac{n - 2}{\sqrt{n (n - 1)}}  \Big( \frac{1}{2}
  ( P_1 + | \mathring{h} |^2 ) | \mathring{h} | - \sqrt{P_1} Q_2 \Big) .
\end{eqnarray*}
We also have
\begin{equation*}  \underset{i \neq j}{\underset{\alpha >1}
{\sum}} ( \mathring{\lambda}_{i} + \mathring{\lambda}_{j}
   ) ( h^{\alpha}_{i j} )^{2} \geq
   \underset{i \neq j}{\underset{\alpha >1}
{\sum}} - \sqrt{2 ( \mathring{\lambda}_{i}^{2} +
   \mathring{\lambda}_{j}^{2} )} ( h^{\alpha}_{i j}
   )^{2} \geq - \sqrt{2P_{1}} Q_{2} . \end{equation*}
Note that $  \tfrac{n-2}{\sqrt{n ( n-1 )}} > \frac{\sqrt{2}}{2}$ if $n
\geq 6$. We get
\begin{equation}
  R_{3} \geq \frac{| H |^{4}}{n^{2}}  + \frac{| H |^{2}}{n} (3P_{1} +P_{2})
  - \tfrac{n-2}{\sqrt{n ( n-1 )}} | H || \mathring{h} |
  \frac{P_1 + | \mathring{h} |^2}{2}  . \label{R3grea}
\end{equation}
For $n \geq 6$, we get from  (\ref{R1less}) and (\ref{R3grea}) that
\begin{equation}\label{R31}
  R_{3} -R_{1} \geq \tfrac{n-2}{\sqrt{n ( n-1 )}} | H | | \mathring{h} | \left(
\frac{P_{2}}{2} - | \mathring{h} |^{2} \right) - | \mathring{h} |^{4} + \frac{1}{n} | \mathring{h} |^{2}
| H |^{2} -2 | \mathring{h} |^{2} P_{2} + \frac{3}{2}  P_{2}^{2}.
\end{equation}

\section{Preservation of curvature pinching}

Let $F: M \times [ 0,T ) \rightarrow \mathbb{S}^{n+q} \left( 1/ \sqrt{c}
\right)$ be a mean curvature flow in a sphere. Let $F_{t} =F ( \cdot ,t )$. We
denote by $M_{t}$ the Riemannian submanifold $F_{t} :M \rightarrow
\mathbb{S}^{n+q} \left( 1/ \sqrt{c} \right)$ at time $t$. Let $\mathcal{H}$ and
$\mathcal{N}$ be two vector bundles over $M \times [ 0,T )$, whose fibers are
given by $\mathcal{H}_{( x,t )} =T_{x} M_{t}$ and $\mathcal{N}_{( x,t )}
=N_{x} M_{t}$. For smooth vector fields $X \in \Gamma ( \mathcal{H} )$ and
$\xi \in \Gamma ( \mathcal{N} )$, we define the covariant time derivatives by
$\nabla_{\partial_{t}} X= ( \overline{\nabla}_{F_{\ast} \partial_{t}}  X )^{\top}$
and $\nabla_{\partial_{t}} \xi = ( \overline{\nabla}_{F_{\ast} \partial_{t}}   \xi
)^{\bot}$.

Without loss of generality, we assume that $c=1$. The following
evolution equations for the mean curvature flow can be found in
{\cite{MR2739807,baker2011mean}}.

\begin{lemma}
  \label{evo}For the mean curvature flow $F: M \times [ 0,T ) \rightarrow
  \mathbb{S}^{n+q}$, we have
  \begin{enumerateroman}
    \item $\nabla_{\partial_{t}} H= \Delta H+n H+H^{\alpha} h^{\alpha}_{i j}
    h_{i j}$,

    \item $\frac{\partial}{\partial t} | h |^{2} = \Delta | h |^{2} -2 | \nabla h |^{2} +2R_{1} +4 |
    H |^{2} -2 n | h |^{2}$,

    \item $\frac{\partial}{\partial t} | H |^{2} = \Delta | H |^{2} -2 | \nabla H |^{2} +2 R_{2} +2n
    | H |^{2}$,

    \item $\frac{\partial}{\partial t} |\mathring{h}|^{2} = \Delta |\mathring{h}|^{2} -2 | \nabla h |^{2} +
    \frac{2}{n} | \nabla H |^{2} +2R_{1} - \frac{2}{n} R_{2} -2 n
    |\mathring{h}|^{2}$.
  \end{enumerateroman}
\end{lemma}

Let $\alpha$ and $\beta$ be two functions given by
\begin{equation*}  \alpha ( x ) =n+ \frac{n}{2 ( n-1 )} x- \frac{n-2}{2 ( n-1 )} \sqrt{x^{2}
   +4 ( n-1 ) x} , \end{equation*}
and
\begin{equation*}  \beta ( x ) = \alpha ( x_{0} ) + \alpha' ( x_{0} ) ( x-x_{0} ) +
   \frac{1}{2} \alpha'' ( x_{0} ) ( x-x_{0} )^{2} , \end{equation*}
where \[x_{0} = \kappa_{n}^{-1} \sqrt{n-1} \left( \sqrt{n-1} - \kappa_{n}
\right)^{2}, \quad \kappa_{n} = \frac{n-4}{2n+2}.\]

We then define a function $\gamma : [ 0,+ \infty ) \rightarrow
\mathbb{R}$ by
\begin{equation}
  \gamma ( x ) = \left\{\begin{array}{ll}
    \alpha ( x ) , & x \geq x_{0} ,\\
    \beta ( x ) , & 0 \leq x<x_{0} .
  \end{array}\right.
\end{equation}
It is obvious that $\gamma$ is a $C^{2}$-function. Moreover, we have the
following

\begin{lemma}\label{gminab}
  For $n \geq 6$ and $x \geq 0$, $\gamma ( x )$ satisfies
  \begin{enumerateroman}
    \item $\gamma ( x ) = \min \{ \alpha ( x ) , \beta ( x ) \}$,

    \item $\gamma ( x ) \geq \gamma ( 0 ) > \frac{3n + 3}{2 n + 6} \sqrt{n-1}$,

    \item $\frac{x}{n-1} +2< \gamma ( x ) < \frac{x}{n-1} + \gamma ( 0 ) <
    \frac{5x}{3n} + \frac{2n}{3}$.
  \end{enumerateroman}
\end{lemma}

\begin{proof}
  By direct computations, for $x>0$, we get
  \begin{equation*}  \alpha' ( x ) = \frac{n}{2 ( n-1 )} - \frac{n-2}{2 ( n-1 )}   \frac{x+2 (
     n-1 )}{\sqrt{x^{2} +4 ( n-1 ) x}} , \end{equation*}
  \begin{equation*}  \alpha'' ( x ) = \frac{2 ( n-2 ) ( n-1 )}{( x^{2} +4 ( n-1 ) x )^{3/2}} >0,\end{equation*}
  \begin{equation*}
  \alpha''' ( x ) =- \frac{6 ( n-2 ) ( n-1 ) ( x+2 ( n-1 ) )}{( x^{2} +4 (
     n-1 ) x )^{5/2}} <0. \end{equation*}
  Then we get
  \[ \alpha ( x_{0} ) = ( \kappa_{n} + \kappa_{n}^{-1} ) \sqrt{n-1} , \;
      \alpha' ( x_{0} ) = \frac{1- \kappa_{n}^{2}}{n-1-
     \kappa_{n}^{2}} , \; \alpha'' ( x_{0} ) = \frac{2 ( n-2 )
     \kappa_{n}^{3}}{( n-1- \kappa_{n}^{2} )^{3} \sqrt{n-1}} .\]

  (i) We have $\alpha ( x_{0} ) = \beta ( x_{0} )$, $\alpha' ( x_{0} ) =
  \beta' ( x_{0} )$ and $\alpha'' ( x_{0} ) = \beta'' ( x_{0} )$. Then from
  $\alpha''' ( x ) <0= \beta''' ( x )$, we get $\alpha ( x ) > \beta ( x )$ for
  $0 \leq x<x_{0}$, and $\alpha ( x ) < \beta ( x )$ for $x>x_{0}$.

  (ii) Let $\iota_{n} = \frac{\sqrt{n - 1} - \kappa_n}{\sqrt{n - 1} + \kappa_n}$.
  We have $\kappa_{n} < \frac{1}{2}$ and $\iota_{n} <1$.
  Then
  \begin{equation}\label{gammaD0}
    \frac{x_{0} \alpha'' ( x_{0} )}{\alpha' ( x_{0} )} = \frac{2
  ( n-2 ) \iota_{n}}{( \kappa_{n}^{-2} -1 ) ( n-1- \kappa_{n}^{2} )} <\frac{2}{3}.
  \end{equation}
  This  implies
  $  \gamma' ( 0 ) = \alpha' ( x_{0} ) -x_{0} \alpha'' ( x_{0} ) >0$.
  Since $\gamma'' ( x ) >0$, we get $\gamma' ( x ) >0$. Thus  $\gamma (
  x ) \geq \gamma ( 0 ) $.
  We have  \[ \frac{\gamma (0)}{\sqrt{n - 1}} = \frac{\alpha (x_0) - x_0 \alpha' (x_0) +
  \frac{1}{2} x_0^2 \alpha'' (x_0)}{\sqrt{n - 1}} = 2 \kappa_n \left( 1 +
  \tfrac{\kappa_n^{- 1} - \kappa_n}{\sqrt{n - 1} + \kappa_n} \right) +
  \tfrac{(n - 2) \kappa_n \iota_n^2}{n - 1 - \kappa_n^2} . \]

  By numerical computations, we get
  $\frac{\gamma (0)}{\sqrt{n - 1}} > \frac{3 n + 3}{2 n + 6}$ for $6
  \leq n \leq 129$.
  Then we consider the case $n \geq 130$. From $\kappa_{n} < \frac{1}{2}$,
  we get $\kappa_n^{- 1} - \kappa_n > \frac{3}{2}$, $
  \tfrac{n - 2}{n - 1 - \kappa_n^2} > \frac{n - 2}{n - 1}$ and $\iota_n^2 > 1 -
  \frac{2}{\sqrt{n - 1} + \kappa_n}$. Thus, $\frac{\gamma (0)}{\sqrt{n - 1}} >
  \frac{\kappa_n}{n - 1} ( 3 n - 4 + \frac{n + 1}{\sqrt{n - 1} + \kappa_n}
  )$. Since $\frac{n + 1}{\sqrt{n - 1} + \kappa_n} > 11$, we have
  $\frac{\gamma (0)}{\sqrt{n - 1}} > \frac{\kappa_n (3 n + 7)}{n - 1} > \frac{3
  n + 3}{2 n + 6}$.

  (iii) Let $\varphi ( x ) = \gamma ( x ) - \frac{x}{n-1}$. Since $\varphi'' (
  x ) >0$ and $\lim_{x \rightarrow \infty} \varphi' ( x ) =0$, we have
  $\varphi' ( x ) <0$. Hence, $2= \varphi ( + \infty ) < \varphi ( x )
  \leq \varphi ( 0 ) = \gamma ( 0 )$. This implies $\frac{x}{n-1}+2<\gamma ( x )\leq\frac{5x}{3n}+\gamma(0)$.

   Since $\gamma ( 0 ) \leq \gamma ( x ) \leq \alpha(x)$, we have
  \begin{equation}\label{gamma0leq}
    \gamma(0)\leq \alpha ( n \sqrt{ n - 1} - 2n + 2 ) =2 \sqrt{n-1} .
  \end{equation}
  If $n \geq 8$, then $2 \sqrt{n-1} < \frac{2n}{3}$.
  If $n=6,\,7$, we obtain $\gamma ( 0 ) < \frac{2n}{3}$ by numerical computations.
  So, we get  $ \gamma ( x ) < \frac{5x}{3n} + \frac{2n}{3}$.
\end{proof}

Set $\mathring{\gamma} ( x ) = \gamma ( x ) -
\frac{x}{n}$. Let $( \cdot )^{+}$ denote the positive part of a function.

\begin{lemma}
  \label{app}For $n \geq 6$ and $x \geq 0$, $\mathring{\gamma} ( x
  )$ satisfies
  \begin{enumerateroman}
    \item $2x  \mathring{\gamma}'' ( x ) + \mathring{\gamma}' ( x ) < \frac{2
    ( n-1 )}{n ( n+2 )} -  \frac{1}{5 n}$,

    \item $x \mathring{\gamma}' ( x ) ( \gamma ( x ) +n ) - \mathring{\gamma}
    ( x ) ( \gamma ( x ) -n ) =0> 2 \mathring{\gamma} ( x ) - \frac{x}{n}
    +x  \mathring{\gamma}' ( x ) $ if $x \geq x_{0}$, $x
    \mathring{\gamma}' ( x ) ( \gamma ( x ) +n ) - \mathring{\gamma} ( x ) (
    \gamma ( x ) -n ) > \frac{1}{6} \big[ \left( 2 \mathring{\gamma} ( x ) -
    \frac{x}{n} +x  \mathring{\gamma}' ( x ) \right)^{+} \big]^{2}$ if $0
    \leq x<x_{0}$,

    \item $\mathring{\gamma} ( x ) -x  \mathring{\gamma}' ( x ) >2$,

    \item $\frac{n-2}{\sqrt{n ( n-1 )}} \sqrt{x  \mathring{\gamma} ( x )} +
    \gamma ( x ) \leq \frac{2}{n} x+n$,

    \item $ \sqrt{\mathring{\gamma} ( x )} - \tfrac{n-2}{4  \sqrt{n ( n-1 )}}
    \sqrt{x}$ is bounded from above.
  \end{enumerateroman}
\end{lemma}

\begin{proof}
  (i) Let $\varphi ( x ) =2x  \mathring{\gamma}'' ( x ) + \mathring{\gamma}' (
  x )$. We have $$\varphi ( x_{0} ) =2x_{0} \alpha'' ( x_{0} ) + \alpha' (
  x_{0} ) - \frac{1}{n} = \frac{n^{2} -2n+2}{2n ( n-1 )} + \frac{n-2}{4 ( n-1
  )} ( \iota_{n}^{3} -3 \iota_{n} ),$$ where $\iota_{n} = \frac{\sqrt{n - 1} - \kappa_n}{\sqrt{n - 1} + \kappa_n}$.
  Since $ \kappa_n <\frac{1}{2}$, we have $1-\frac{1}{\sqrt{n-1}}<\iota_n<1$.
  Then we get  $\iota_{n}^{3} -3
  \iota_{n} < \frac{4}{n} -2$. Thus we obtain $\varphi ( x_{0} )<\frac{1}{n} < \frac{2 (
  n-1 )}{n ( n+2 )}-\frac{1}{5 n}$.

  If $x \geq x_{0}$, then $\varphi' ( x ) =- \frac{6 ( n-2 ) ( n-1
  )}{\sqrt{x} ( x+4 ( n-1 ) )^{5/2}} <0$. If $0 \leq x<x_{0}$, then
  $\varphi' ( x ) =3 \alpha'' ( x_{0} ) >0$. Hence, $\varphi ( x )
  \leq \varphi ( x_{0} ) < \frac{2 ( n-1 )}{n ( n+2 )}-\frac{1}{5 n}$.

  (ii) We can check that $\alpha$ and $\alpha'$ satisfies
  \begin{equation}
    ( \alpha ( x ) +n ) x  \alpha' ( x ) =2 x+ \alpha ( x )(\alpha ( x ) -n
    ) . \label{alid1}
  \end{equation}
  Let $$\psi_{1} ( x ) =x \mathring{\gamma}' ( x ) ( \gamma ( x ) +n ) -
  \mathring{\gamma} ( x ) ( \gamma ( x ) -n ),\quad \psi_{2} ( x ) =2
  \mathring{\gamma} ( x ) - \frac{x}{n} +x  \mathring{\gamma}' ( x ).$$ From
  (\ref{alid1}), we get $\psi_{1} ( x ) =0$ if $x \geq x_{0}$.

  We have $\psi_{2}' ( x ) =3 \alpha' ( x ) +x  \alpha'' ( x ) - \frac{4}{n}$
  if $x \geq x_{0}$. Then $\lim_{x \rightarrow \infty}   \psi_{2}' ( x )
  =- \frac{n-4}{n ( n-1 )}$. Noting that
  \begin{equation*}  \psi_{2}'' ( x ) = \frac{2 ( n-2 ) ( n-1 ) ( x^{2} +10 ( n-1 ) x )}{(
     x^{2} +4 ( n-1 ) x )^{5/2}} >0 \quad \text{for} \  x
     \geq x_{0} , \end{equation*}
  we have $\psi_{2}' ( x ) <0$ if $x \geq x_{0}$. Hence  $\psi_{2}
  ( x ) \leq \psi_{2} ( x_{0} ) $ for $x \geq x_{0}$.
  We have $$\frac{\psi_2 (x_0)}{x_0}=
  \frac{2 \alpha (x_0)}{x_0} + \alpha' (x_0) -  \frac{4}{n} =
  \frac{2 (1+\kappa _n^2)}{(\sqrt{n-1}-\kappa _n)^2}
  +\frac{1-\kappa _n^2}{n-1-\kappa _n^2}-\frac{4}{n}.$$
  By numerical computations, we get $\frac{\psi_2 (x_0)}{x_0}<0$ for $6 \leq n \leq 39$.
  If $n \geq 40$, it follows from $\kappa_n<\frac{1}{2}$ that $(\sqrt{n-1}-\kappa _n)^2>\frac{4}{5}(n-1-\kappa _n^2)$.
  Thus  $\frac{\psi_2 (x_0)}{x_0}< \frac{(3 n+8) \kappa _n^2-n+8}{2 n (n-1-\kappa _n^2)} <0$.
  Hence we have $$\psi_{2}  ( x ) \leq \psi_{2} ( x_{0} ) <0 \quad \text{for} \ x \geq x_{0}.$$

  From the $C^{2}$-continuity of $\gamma$ and (\ref{alid1}), we get $\psi_{1}
  ( x_{0} ) = \psi_{1}' ( x_{0} ) =0$. Calculating derivatives, we get
  \begin{equation*}  \psi_{1}'' ( x ) =3 \alpha'' ( x_{0} ) [ \alpha'' ( x_{0} ) x^{2} + (
     \alpha' ( x_{0} ) -x_{0} \alpha'' ( x_{0} ) ) x+n ] \quad
     \operatorname{for} \; 0 \leq x \leq x_{0} . \end{equation*}
  It follows from \eqref{gammaD0} that $\psi_1'' (x) \geq \alpha''
  (x_0) (\alpha' (x_0) x + 3 n)$.     If $0 \leq x \leq x_{0}$,
   we have $$\psi_1' (x) = \int^x_{x_0} \psi_1'' (t) \mathrm{d} t
  \leq \mu_n (x - x_0), \quad \psi_1 (x) = \int^x_{x_0} \psi_1' (t) \mathrm{d}
  t \geq \frac{\mu_n}{2} (x - x_0)^2,$$ where
  $\mu_n = \alpha'' (x_0) ( \frac{1}{2}x_0 \alpha' (x_0)  + 3 n )$.  Set $\psi_3 (x) =
  \sqrt{3 \mu_n} (x_0 - x) - \psi_2 (x)$ for  $0 \leq x\leq x_0$.
  Then  $\sqrt{6\psi_1 (x)} - \psi_2 (x) \geq \psi_3 (x)$. Since
  $\psi_3'' (x) = - \psi_2'' (x_0) < 0$, we get $\min_{0 \leq x \leq x_0} \psi_3
  (x) = \min \{ \psi_3 (0), \psi_3 (x_0) \}$.
  We have obtained $\psi_3 (x_0) = -
  \psi_2 (x_0) > 0$. By numerical computations, we get $\psi_3 (0) > 0$ for $6\leq n
  \leq 33$. If $n\geq 34$, we have $\frac{3}{7}\leq \kappa_n<\frac{1}{2}$.
  Hence, we get $\mu_n > 3 n \alpha'' (x_0) > \frac{6 (n - 2) ( \frac{3}{7}
  )^3}{(n - 1)^{5 / 2}} > \frac{64}{27 (n - 1)^2}$ and $x_0 > 2
  \sqrt{n - 1} ( \sqrt{n - 1} - \frac{1}{2} )^2 > \frac{3}{2}
  (n - 1)^{\frac{3}{2}}$.
  Combining these and \eqref{gamma0leq}, we get $\psi_3 (0) =
  \sqrt{3\mu_n} x_0 - 2 \gamma (0) > 0$.
  This yields $\min_{0 \leq x \leq x_0} \psi_3 (x) > 0$.
  Therefore, we obtain $\psi_1 (x) >\frac{1}{6} [\psi_2^+
  (x)]^2$ for $0 \leq x < x_0$.

  (iii) Noting that $( \alpha ( x ) -x \alpha' ( x ) )' =- \alpha'' ( x ) x<0$
  and $( \beta ( x ) -x \beta' ( x ) )' =- \alpha'' ( x_{0} ) x \leq 0$,
  we have $\left( \mathring{\gamma} ( x ) -x  \mathring{\gamma}' ( x )
  \right)' <0$. This together with $\lim_{x \rightarrow \infty} ( \alpha ( x )
  -x \alpha' ( x ) ) =2$ implies $\mathring{\gamma} ( x ) -x
  \mathring{\gamma}' ( x ) >2$.

  (iv) We can check that $\alpha$ satisfies
  $$ \tfrac{n-2}{\sqrt{n ( n-1 )}} \sqrt{x  ( \alpha ( x ) - \tfrac{x}{n}
    )} + \alpha ( x ) = \tfrac{2x}{n} +n. $$
  Combining this identity and the inequality $\gamma ( x ) \leq \alpha ( x )$, we prove (iv).

  (v) If $n=6$, we get $\lim_{x \rightarrow \infty} \big(
  \sqrt{\mathring{\gamma} ( x )} - \tfrac{n-2}{4  \sqrt{n ( n-1 )}} \sqrt{x}
  \big) =0$.

  If $n \geq 7$, we have $\lim_{x \rightarrow \infty}
  \sqrt{\mathring{\gamma} ( x ) /x} = \frac{1}{\sqrt{n ( n-1 )}} <
  \frac{n-2}{4 \sqrt{n ( n-1 )}}$. Thus $ \sqrt{\mathring{\gamma} ( x )} -
  \tfrac{n-2}{4  \sqrt{n ( n-1 )}} \sqrt{x}$ is negative when $x$ is large
  enough.
\end{proof}

Let $\omega : [ 0,+ \infty ) \rightarrow \mathbb{R}$ be a positive
$C^{2}$-function, and let
\begin{equation*}  \omega ( x ) = \frac{x^{2}}{\sqrt{x^{2} +4 ( n-1 ) x}} \left[ \left( 1+
   \tfrac{n^{2}}{x} \right) \tfrac{\frac{n}{n-2} - ( 1+4 ( n-1 ) /x^{}
   )^{-1/2}}{\frac{n}{n-2} +  ( 1+4 ( n-1 ) /x )^{-1/2}} \right]^{2}
   \quad \operatorname{for} \; x \geq x_{0} . \end{equation*}
We can check that $\omega$ satisfies the following equation
\begin{equation}
  \frac{x \omega' ( x )}{\omega ( x )} = \frac{2  \alpha ( x ) -x  \alpha' (
  x ) -3 n}{\alpha ( x ) +n} \quad \operatorname{for} \; x \geq x_{0}. \label{dlnw}
\end{equation}

For a small positive number $\varepsilon$, we set
$\mathring{\gamma}_{\varepsilon} ( x ) = \mathring{\gamma} ( x ) - \varepsilon
\omega ( x )$.

\begin{lemma}
  \label{gaep}Let $n \geq 6$. There exists a positive constant $\varepsilon_1$, such that
  for all $\varepsilon \in (0, \varepsilon_1)$ and $x \geq 0$, we have
  \begin{enumerateroman}
    \item $2x  \mathring{\gamma}_{\varepsilon}'' ( x ) +
    \mathring{\gamma}_{\varepsilon}' ( x ) < \frac{2 ( n-1 )}{n ( n+2 )}$,

    \item $x \mathring{\gamma}_{\varepsilon}' ( x ) \left(
    \mathring{\gamma}_{\varepsilon} ( x ) + \frac{x}{n} +n \right) -
    \mathring{\gamma}_{\varepsilon} ( x ) \left(
    \mathring{\gamma}_{\varepsilon} ( x ) + \frac{x}{n} -n \right) >
    \frac{1}{6} \big[ \left( 2 \mathring{\gamma}_{\varepsilon} ( x ) -
    \frac{x}{n} +x  \mathring{\gamma}_{\varepsilon}' ( x ) \right)^{+}
    \big]^{2}$.
  \end{enumerateroman}
\end{lemma}

\begin{proof}
  (i) We have $$2x  \mathring{\gamma}_{\varepsilon}'' ( x ) +
  \mathring{\gamma}_{\varepsilon}' ( x ) =2x  \mathring{\gamma}'' ( x ) +
  \mathring{\gamma}' ( x ) - \varepsilon ( 2x  \omega'' ( x ) + \omega' ( x )
  ).$$ We figure out that $\lim_{x \rightarrow \infty}   \omega' ( x ) = ( n-1
  )^{-2}$ and $\lim_{x \rightarrow \infty}  x \omega'' ( x ) =0$. Thus $2x
  \omega'' ( x ) + \omega' ( x )$ is bounded. From Lemma \ref{app} (i), we
  obtain the conclusion.

  (ii)  We have
  \begin{eqnarray}
    &  & x \mathring{\gamma}_{\varepsilon}' ( x ) \left(
    \mathring{\gamma}_{\varepsilon} ( x ) + \tfrac{x}{n} +n \right) -
    \mathring{\gamma}_{\varepsilon} ( x ) \left(
    \mathring{\gamma}_{\varepsilon} ( x ) + \tfrac{x}{n} -n \right) \nonumber\\
    & = & x \mathring{\gamma}' ( x ) \left( \mathring{\gamma} ( x ) +
    \tfrac{x}{n} +n \right) - \mathring{\gamma} ( x ) \left( \mathring{\gamma}
    ( x ) + \tfrac{x}{n} -n \right) \nonumber\\
    &  & + \varepsilon   \omega ( x ) \left[ 2 \mathring{\gamma} ( x ) -x
    \mathring{\gamma}' ( x ) + \tfrac{x}{n} -n- \left( \mathring{\gamma} ( x )
    + \tfrac{x}{n} +n \right) \frac{x  \omega' ( x )}{\omega ( x )} \right]
    \nonumber\\
    &  & + \varepsilon^{2} \omega ( x ) ( x  \omega' ( x ) - \omega ( x ) )\nonumber
  \end{eqnarray}
  and
  \[ \left( 2 \mathring{\gamma}_{\varepsilon} (x) - \tfrac{x}{n} + x
   \mathring{\gamma}_{\varepsilon}' (x) \right)^+ \leq \left( 2 \mathring{\gamma}
   (x) - \tfrac{x}{n} + x \mathring{\gamma}' (x)
    \right)^+ + \varepsilon (-x \omega'  (x))^+ . \]
  From \eqref{dlnw}, there exists a small positive number
  $\delta$, such that if $x>x_0-\delta$, then
  \begin{equation*}  2 \mathring{\gamma} ( x ) -x  \mathring{\gamma}' ( x ) + \tfrac{x}{n} -n-
    \left( \mathring{\gamma} ( x ) + \tfrac{x}{n} +n \right) \frac{x  \omega'
    ( x )}{\omega ( x )} >n  . \end{equation*}
  From Lemma \ref{app} (ii),  if $x>x_0-\delta$, we have \[ 2 \mathring{\gamma}
  (x) - \tfrac{x}{n} + x \mathring{\gamma}' (x) <0.\]
  Hence, if $x>x_0-\delta$, we have
  \begin{eqnarray*}
    &  & x \mathring{\gamma}_{\varepsilon}' ( x ) \left(
      \mathring{\gamma}_{\varepsilon} ( x ) + \tfrac{x}{n} +n \right) -
      \mathring{\gamma}_{\varepsilon} ( x ) \left(
      \mathring{\gamma}_{\varepsilon} ( x ) + \tfrac{x}{n} -n \right) -
      \tfrac{1}{6} \big[ \left( 2 \mathring{\gamma}_{\varepsilon} ( x ) -
      \tfrac{x}{n} +x  \mathring{\gamma}_{\varepsilon}' ( x ) \right)^{+}
      \big]^{2}\\
    & &\ \geq\, n \varepsilon \omega (x) + \varepsilon^2 \omega (x) (x
    \omega' (x) - \omega (x)) - \frac{\varepsilon^2}{6} [(- x \omega' (x))^+]^2
    .
  \end{eqnarray*}
  Notice that $x  \omega' ( x ) - \omega ( x )$ and $( -x
  \omega' ( x ) )^{+}$ are bounded. Thus, the right hand side of the above inequality
  is positive if $\varepsilon$ is small enough.

  For $0 \leq x \leq x_0-\delta$, since $\mathring{\gamma}_{\varepsilon} \rightarrow \mathring{\gamma},\,
  \mathring{\gamma}_{\varepsilon}'\rightarrow\mathring{\gamma}'$ uniformly as $\varepsilon\rightarrow 0$,
   the assertion follows from Lemma  \ref{app} (ii).
\end{proof}

For convenience, we denote $\gamma ( | H |^{2} )$, $\mathring{\gamma} ( | H
|^{2} )$, $\mathring{\gamma}_{\varepsilon} ( | H |^{2} )$ and $\omega ( | H
|^{2} )$ by $\gamma$, $\mathring{\gamma}$, $\mathring{\gamma}_{\varepsilon}$
and $\omega$, respectively. We use $( \cdot )'$ to denote the derivative with
respect to $| H |^{2}$.

Consider the mean curvature flow $F: M \times [ 0,T ) \rightarrow
\mathbb{S}^{n+q}$. Suppose $M_{0}$ is an $n$-dimensional $( n \geq 6 )$
closed submanifold satisfying $| h |^{2} < \gamma$. Since $M_{0}$ is compact,
there exists a small positive number $\varepsilon$, such that $M_{0}$
satisfies $|\mathring{h}|^{2} < \mathring{\gamma}_{\varepsilon}$.

Now we show that the pinching condition is preserved.

\begin{theorem}
  \label{pinch}If $M_{0}$ satisfies $|\mathring{h}|^{2} <
  \mathring{\gamma}_{\varepsilon}$, then this condition holds for all time $t
  \in [ 0,T )$.
\end{theorem}

\begin{proof}
  Let $U= |\mathring{h}|^{2} - \mathring{\gamma}_{\varepsilon}$.
    By Lemmas \ref{dA2},  \ref{R12} and \ref{evo}, we have
  \begin{eqnarray*}
    \Big( \frac{\partial}{\partial t} - \Delta \Big) U & = & -2 | \nabla h |^{2} + \frac{2}{n} |
    \nabla H |^{2} +2 \mathring{\gamma}'_{\varepsilon}   | \nabla H |^{2} +
    \mathring{\gamma}''_{\varepsilon}   \big| \nabla | H |^{2} \big|^{2}\\
    &  & +2R_{1} - \frac{2}{n} R_{2} -2 n |\mathring{h}|^{2} -2
    \mathring{\gamma}'_{\varepsilon} \cdot ( R_{2} +n | H |^{2} ) \\
    & \leq & 2\Big[ - \frac{2 ( n-1 )}{n (
    n+2 )} + \mathring{\gamma}'_{\varepsilon} +2 | H |^{2}
    \mathring{\gamma}''_{\varepsilon} \Big] | \nabla H |^{2}\\
    &  & +2 |\mathring{h}|^{2} \Big( |\mathring{h}|^{2} + \frac{1}{n} | H |^{2} - n
    \Big) +2P_{2} \Big( 2 |\mathring{h}|^{2} - \frac{1}{n} | H |^{2} - \frac{3}{2}
    P_{2} \Big)\\
    &  & -2 \mathring{\gamma}'_{\varepsilon} \cdot | H |^{2} \Big(
    |\mathring{h}|^{2} + \frac{1}{n} | H |^{2} +n \Big) +2
    \mathring{\gamma}'_{\varepsilon} \cdot | H |^{2} P_{2} .
  \end{eqnarray*}
  From Lemma \ref{gaep} (i), the coefficient of $| \nabla H |^{2}$ is
  negative. Replacing $|\mathring{h}|^{2}$ by $U+ \mathring{\gamma}_{\varepsilon}$,
  the above formula becomes
  \begin{eqnarray*}
    \Big( \frac{\partial}{\partial t} - \Delta \Big) U & \leq & 2U \Big( 2
    \mathring{\gamma}_{\varepsilon} + \frac{1}{n} | H |^{2} -n- | H |^{2}
    \mathring{\gamma}'_{\varepsilon} +2P_{2} \Big) +2U^{2}\\
    &  & +2 \mathring{\gamma}_{\varepsilon} \cdot \Big(
    \mathring{\gamma}_{\varepsilon} + \frac{1}{n} | H |^{2} -n \Big) -2 | H
    |^{2}   \mathring{\gamma}'_{\varepsilon} \cdot \Big(
    \mathring{\gamma}_{\varepsilon} + \frac{1}{n} | H |^{2} +n \Big)\\
    &  & +2P_{2} \Big( 2 \mathring{\gamma}_{\varepsilon} - \frac{1}{n} | H
    |^{2} + | H |^{2} \mathring{\gamma}'_{\varepsilon} - \frac{3}{2} P_{2}
    \Big) .
  \end{eqnarray*}
  By Lemma \ref{gaep} (ii), we have
  \begin{eqnarray*}
    &  & P_2 \Big( 2 \mathring{\gamma}_{\varepsilon} - \frac{1}{n} | H |^2 + |
    H |^2 \mathring{\gamma}'_{\varepsilon} - \frac{3}{2} P_2 \Big)\\
    & \leq & \frac{1}{6} \Big[ \Big( 2 \mathring{\gamma}_{\varepsilon} -
    \frac{1}{n} | H |^2 + | H |^2 \mathring{\gamma}'_{\varepsilon} \Big)^+
    \Big]^2\\
    & \leq & | H |^2
    \mathring{\gamma}'_{\varepsilon} \cdot \Big(
    \mathring{\gamma}_{\varepsilon} + \frac{1}{n} | H |^2 + n \Big)
    -\mathring{\gamma}_{\varepsilon} \cdot \Big(
      \mathring{\gamma}_{\varepsilon} + \frac{1}{n} | H |^2 - n \Big) .
  \end{eqnarray*}
   Then the assertion follows from the maximum  principle.
\end{proof}

\section{An estimate for traceless second fundamental form}

In this section, we derive the following estimate for the traceless second
fundamental form.

\begin{theorem}
  \label{sa0h2}If $M_{0}$ satisfies $|\mathring{h}|^{2} <
  \mathring{\gamma}_{\varepsilon}$, then there exist constants $0< \sigma
  \leq \varepsilon^{5/2}$, and $C_{0} >0$ depending only on $M_{0}$, such
  that for all $t \in [ 0,T )$ we have
  \begin{equation*}  |\mathring{h}|^{2} \leq C_{0}   ( | H |^{2} +1 )^{1- \sigma} \mathrm{e}^{-2
     \sigma t} . \end{equation*}
\end{theorem}

To prove this theorem, we will show that the following function decays exponentially.
\begin{equation*}  f_{\sigma} = \frac{|\mathring{h}|^{2}}{\mathring{\gamma}^{1- \sigma}} ,
\end{equation*}   where $0< \sigma <\frac{1}{2}$.

For the time derivative of $f_{\sigma}$, we obtain

\begin{lemma}
  \label{ptf}There exist positive constants $C_{1},\tilde{C}_1$ depending only on $n$,
  such that
  \begin{equation*}  \frac{\partial}{\partial t} f_{\sigma} \leq \Delta f_{\sigma} + \frac{4C_{1}}{|\mathring{h}|} |
     \nabla f_{\sigma} | | \nabla H | - \frac{2
     f_{\sigma}}{5n |\mathring{h}|^{2}} | \nabla H |^{2} +2f_{\sigma} \Big( 3 \sigma |
     h |^{2} - \sigma n- \varepsilon \tilde{C}_1 \Big) . \end{equation*}
\end{lemma}

\begin{proof}
  By a straightforward computation, we get
  \begin{equation*}  \frac{\partial}{\partial t} f_{\sigma} =f_{\sigma} \left( \frac{\frac{\partial}{\partial t} |\mathring{h}|^{2}}{|\mathring{h}|^{2}} - (
     1- \sigma ) \frac{\frac{\partial}{\partial t} \mathring{\gamma}}{\mathring{\gamma}} \right) , \end{equation*}
  and
  \begin{equation}
    \Delta f_{\sigma} =f_{\sigma} \left( \frac{\Delta
    |\mathring{h}|^{2}}{|\mathring{h}|^{2}} - ( 1- \sigma ) \frac{\Delta
    \mathring{\gamma}}{\mathring{\gamma}} \right) -2 ( 1- \sigma )
    \frac{\left\langle \nabla f_{\sigma} , \nabla \mathring{\gamma}
    \right\rangle}{\mathring{\gamma}} + \sigma ( 1- \sigma ) f_{\sigma}
    \frac{\left| \nabla \mathring{\gamma} \right|^{2}}{\left|
    \mathring{\gamma} \right|^{2}} . \label{lapf}
  \end{equation}
  From the evolution equations, we have
  \begin{eqnarray}
    &\left( \frac{\partial}{\partial t} - \Delta \right) f_{\sigma} & \leq \ 2
    \frac{| \nabla f_{\sigma} || \nabla \mathring{\gamma}
    |}{\mathring{\gamma}}  \nonumber\\
    &  & +f_{\sigma} \left[ \frac{2}{|\mathring{h}|^{2}} \left( \frac{| \nabla H
    |^{2}}{n} - | \nabla h |^{2} \right)
     + \frac{1- \sigma}{\mathring{\gamma}} \left( 2
    \mathring{\gamma}' | \nabla H |^{2} + \mathring{\gamma}''   \big| \nabla | H
    |^{2} \big|^{2} \right) \right]  \label{dtf}\\
    &  & +2 f_{\sigma} \left[ \frac{1}{|\mathring{h}|^{2}} \left( R_{1} -
    \frac{1}{n}  R_{2} \right) -n- ( 1- \sigma )
    \frac{\mathring{\gamma}'}{\mathring{\gamma}} (  R_{2} +n | H |^{2} )
    \right] . \nonumber
  \end{eqnarray}

  By the definition of $\mathring{\gamma}$, there exists a
  positive constant $C_{1}$, such that $\frac{| H | |
  \mathring{\gamma}' | }{ \sqrt{\mathring{\gamma}}} <C_{1}$. Then
  \begin{equation}
    \frac{\left| \nabla \mathring{\gamma} \right|}{\mathring{\gamma}}
    \leq \frac{2}{\mathring{\gamma}} \left| \mathring{\gamma}' \right| |
    H | | \nabla H | \leq \frac{2C_{1}}{\sqrt{\mathring{\gamma}}} | \nabla H |
    \leq \frac{2C_{1}}{|\mathring{h}|} | \nabla H | .
    \label{aopaodfsdh2}
  \end{equation}

  By Lemma \ref{dA2} and Lemma \ref{app} (i),
    we have the following estimate for the expression in the first square bracket of the right
  hand side of (\ref{dtf}).
  \begin{eqnarray*}
    &  & \frac{2}{|\mathring{h}|^{2}} \left( \frac{| \nabla H |^{2}}{n} - | \nabla h
    |^{2} \right) + \frac{1- \sigma}{\mathring{\gamma}} \left( 2
    \mathring{\gamma}' | \nabla H |^{2} + \mathring{\gamma}'' \big| \nabla | H
    |^{2} \big|^{2} \right)\\
    & \leq &\frac{2}{|\mathring{h}|^2} \cdot \frac{2 (1 - n)}{n (n + 2)}  | \nabla H |^2 +
    \frac{1 - \sigma}{\mathring{\gamma}} \left( 2 \mathring{\gamma}' + 4
    \mathring{\gamma}''  | H |^2 \right) | \nabla H |^2\\
    & \leq & \frac{2}{| \mathring{h} |^{2}} \cdot \frac{2 ( 1-n )}{n ( n+2 )}   | \nabla H |^{2} +
    \frac{1}{| \mathring{h} |^{2}}  \left( \frac{4 ( n-1 )}{n ( n+2 )} -
    \frac{2}{5n} \right) | \nabla H |^{2}\\
    & \leq & - \frac{2 }{5n |\mathring{h}|^{2}} | \nabla H |^{2} .
  \end{eqnarray*}

  We then estimate the expression in the second square bracket of the right
  hand side of (\ref{dtf}). By Lemma \ref{R12}, we have
  \begin{eqnarray}
    &  & \frac{1}{| \mathring{h}|^2} \left( R_1 - \frac{1}{n} R_2
    \right) - n - (1 - \sigma) \frac{\mathring{\gamma}'}{\mathring{\gamma}} (R_2
    + n | H |^2) \nonumber\\
    \qquad & \le & | h |^2 - n + 2 P_2 - \frac{P_2}{| \mathring{h}
        |^2} \left( \frac{3 P_2}{2} + \frac{| H |^2}{n} \right) - (1 - \sigma)
    \frac{\mathring{\gamma}'}{\mathring{\gamma}} | H |^2 ( | h |^2 + n - P_2)
    \label{sH22c}\\
    & \le & | h |^2 - n + 2 P_2 - (1 - \sigma) \left[
    \frac{P_2}{\mathring{\gamma}} \left( \frac{3 P_2}{2} + \frac{| H |^2}{n}
    \right) + \frac{\mathring{\gamma}'}{\mathring{\gamma}} | H |^2 ( | h |^2 + n
    - P_2) \right] \nonumber\\
    & = & \sigma (| h |^2 + n + 2 P_2) - 2 n + (1 - \sigma) \times
    \nonumber\\
    &  & \left[ | h |^2 + n + 2 P_2 - \frac{P_2}{\mathring{\gamma}} \left(
    \frac{3 P_2}{2} + \frac{| H |^2}{n} \right) -
    \frac{\mathring{\gamma}'}{\mathring{\gamma}} | H |^2 ( | h |^2 + n - P_2)
    \right] . \nonumber
  \end{eqnarray}
  By Lemma \ref{app} (ii), (iii), we get
  \begin{eqnarray*}
    &  & | h |^2 + n + 2 P_2 - \frac{P_2}{\mathring{\gamma}} \left( \frac{3
        P_2}{2} + \frac{| H |^2}{n} \right) -
    \frac{\mathring{\gamma}'}{\mathring{\gamma}} | H |^2 ( | h |^2 + n - P_2)\\
    & = & (| h |^2 + n) \left( 1 - \frac{
        \mathring{\gamma}'}{\mathring{\gamma}} | H |^2 \right) +
    \frac{P_2}{\mathring{\gamma}} \left( 2 \mathring{\gamma} - \frac{| H |^2}{n}
    + \mathring{\gamma}'  | H |^2 - \frac{3 P_2}{2} \right)\\
    & \le & (\gamma - \varepsilon \omega + n) \frac{\mathring{\gamma} -
        \mathring{\gamma}'  | H |^2}{\mathring{\gamma}} + \frac{1}{6
        \mathring{\gamma}} \left[ \left( 2 \mathring{\gamma} - \tfrac{1}{n} | H |^2 +
    \mathring{\gamma}' | H |^2 \right)^+ \right]^2\\
    & \leq & - \frac{2 \varepsilon \omega}{\mathring{\gamma}} +
    \frac{1}{\mathring{\gamma}} \left[ (\gamma + n) \left( \mathring{\gamma} -
    \mathring{\gamma}'  | H |^2 \right) + \frac{1}{6} \left[ \left( 2
    \mathring{\gamma} - \tfrac{1}{n} | H |^2 + \mathring{\gamma}' | H |^2
    \right)^+ \right]^2 \right]\\
    & \leq & - \frac{2 \varepsilon \omega}{\mathring{\gamma}} + 2 n.
  \end{eqnarray*}
  Hence, the right hand side of (\ref{sH22c}) is less than
  \begin{equation*}
    \sigma ( | h |^{2} +n +2  P_{2})  -2n
    + ( 1- \sigma ) \Big( 2n- \frac{2 \varepsilon \omega}{\mathring{\gamma}}
     \Big) < 3\sigma   | h |^{2} - \sigma n- \frac{ \varepsilon \omega}{\mathring{\gamma}} .
  \end{equation*}
  Let $\tilde{C}_1$ be a positive constant such that $\tilde{C}_1 <
  \frac{ \omega}{\mathring{\gamma}}$.
  Then we complete the proof of Lemma \ref{ptf}.
\end{proof}

We need the following estimate for the the Laplacian of $|\mathring{h}|^{2}$.

\begin{lemma}
  \label{lapa0}Suppose that $M$ is an n-dimensional $( n \geq 6 )$
  submanifold in $\mathbb{S}^{n+q}$ satisfying $|\mathring{h}|^{2} <
  \mathring{\gamma}_{\varepsilon}$.  There exists a positive constant
  $C_{2}$ depending only on $n$, such that for sufficiently small $\sigma >0$,
  there holds
  \[  \Delta |\mathring{h}|^{2} \geq 2 \big\langle \mathring{h} , \nabla^{2} H
     \big\rangle +2 \varepsilon C_{2}   |\mathring{h}|^{2} \Big( | h |^{2} -
     \frac{\varepsilon\tilde{C}_1}{3 \sigma} \Big) . \]
\end{lemma}

\begin{proof}
  By (\ref{lapho2}) and Lemma \ref{dA2} (i), we have
  $  \frac{1}{2} \Delta |\mathring{h}|^{2} \geq  \big\langle \mathring{h} , \nabla^{2} H
    \big\rangle +n |\mathring{h}|^{2} -R_{1} +R_{3} $.

   It follows from \eqref{R31} and Lemma \ref{app} (iv)
  that
  \begin{eqnarray*}
    n |\mathring{h}|^{2} -R_{1} +R_{3} & \geq & |\mathring{h}|^{2} \left( n -
    \tfrac{n-2}{\sqrt{n ( n-1 )}} | H | |\mathring{h}| - |\mathring{h}|^{2} + \frac{1}{n} |
    H |^{2} \right)\\
    &  & +P_{2} \left( \tfrac{n-2}{2 \sqrt{n ( n-1 )}} | H | |\mathring{h}| -2
    |\mathring{h}|^{2} + \frac{3}{2}  P_{2} \right)\\
    & \geq & |\mathring{h}|^{2} \left( n -  \tfrac{n-2}{\sqrt{n ( n-1 )}} | H |
    \sqrt{\mathring{\gamma}} - \left( \mathring{\gamma} - \varepsilon \omega
    \right) + \frac{1}{n} | H |^{2} \right)\\
    &  & -P_{2} \left( |\mathring{h}| \Big( 2 \sqrt{\mathring{\gamma}} -
    \tfrac{n-2}{2  \sqrt{n ( n-1 )}} | H | \Big) - \frac{3}{2}  P_{2}
    \right)\\
    & \geq & |\mathring{h}|^{2}   \varepsilon   \omega - \frac{1}{6}
    |\mathring{h}|^{2} \left[ \Big( 2 \sqrt{\mathring{\gamma}} - \tfrac{n-2}{2
    \sqrt{n ( n-1 )}} | H | \Big)^{+} \right]^{2} .
  \end{eqnarray*}
  Let $C_{2}$ be a positive constant such that $C_2 < \frac{\omega}{\gamma}$. Thus, $ \omega
  >  C_{2} | h |^{2}$. From Lemma \ref{app} (v), if $\sigma$ is small enough, then
  $\frac{\varepsilon^{2} }{ \sigma} > \frac{1}{2\tilde{C}_{1}C_{2}} \Big[  \Big( 2
  \sqrt{\mathring{\gamma}} - \tfrac{n-2}{2  \sqrt{n ( n-1 )}} | H | \Big)^+
  \Big]^{2}$. Therefore, we obtain the conclusion.
\end{proof}

From (\ref{lapf}), (\ref{aopaodfsdh2}) and Lemma \ref{lapa0}, we have
\begin{eqnarray*}
  \Delta f_{\sigma} & \geq & \frac{\Delta
  |\mathring{h}|^{2}}{\mathring{\gamma}^{1- \sigma}} -  ( 1- \sigma )
  \frac{f_{\sigma}}{\mathring{\gamma}} \Delta \mathring{\gamma} -2 ( 1- \sigma
  ) \frac{\left\langle \nabla f_{\sigma} , \nabla \mathring{\gamma}
  \right\rangle}{\mathring{\gamma}}\\
  & \geq & \frac{2 \langle
  \mathring{h} , \nabla^{2} H \rangle }{\mathring{\gamma}^{1- \sigma}}+2 \varepsilon C_{2}   \Big( | h
  |^{2} - \frac{\varepsilon\tilde{C}_1}{3 \sigma} \Big) f_{\sigma}
   -  ( 1- \sigma ) \frac{f_{\sigma}}{\mathring{\gamma}} \Delta
  \mathring{\gamma} - \frac{4C_{1}}{|\mathring{h}|} | \nabla f_{\sigma} | | \nabla H
  | .
\end{eqnarray*}
This is equivalent to
\begin{equation}
  2 \varepsilon C_{2} \Big( | h |^{2} - \frac{\varepsilon\tilde{C}_1}{3 \sigma}
  \Big) f_{\sigma} \leq \Delta f_{\sigma} +  ( 1- \sigma )
  \frac{f_{\sigma}}{\mathring{\gamma}} \Delta \mathring{\gamma} -
  \frac{2  \langle \mathring{h} ,
  \nabla^{2} H \rangle }{\mathring{\gamma}^{1- \sigma}}+ \frac{4C_{1}}{|\mathring{h}|} | \nabla f_{\sigma} | |
  \nabla H | . \label{eC2f}
\end{equation}
We multiply both sides of the above inequality by
$f_{\sigma}^{p-1}$, then integrate them over $M_{t}$. Using the
divergence theorem, we have
\begin{equation}
  \int_{M_{t}} f_{\sigma}^{p-1} \Delta f_{\sigma} \mathrm{d} \mu_{t} =- ( p-1 )
  \int_{M_{t}} f_{\sigma}^{p-2} | \nabla f_{\sigma} |^{2} \mathrm{d} \mu_{t}. \label{eC2f1}
\end{equation}
From (\ref{aopaodfsdh2}), we have
\begin{eqnarray}
  \int_{M_{t}} \frac{f_{\sigma}^{p}}{\mathring{\gamma}} \Delta
  \mathring{\gamma}   \mathrm{d} \mu_{t} & = & - \int_{M_{t}} \left\langle \nabla
  \left( \frac{f_{\sigma}^{p}}{\mathring{\gamma}} \right) , \nabla
  \mathring{\gamma} \right\rangle \mathrm{d} \mu_{t} \nonumber\\
  & = & \int_{M_{t}} \left( - \frac{p f_{\sigma}^{p-1}}{\mathring{\gamma}}
  \left\langle \nabla f_{\sigma} , \nabla \mathring{\gamma} \right\rangle +
  \frac{f_{\sigma}^{p}}{\mathring{\gamma}^{2}} \left| \nabla \mathring{\gamma}
  \right|^{2} \right) \mathrm{d} \mu_{t} \\
  & \leq & \int_{M_{t}} \left( \frac{2p C_{1} f_{\sigma}^{p-1}}{| \mathring{h} |}
  | \nabla f_{\sigma} | |   \nabla H | + \frac{4C_{1}^{2}
  f_{\sigma}^{p}}{| \mathring{h} |^{2}} | \nabla H |^{2} \right) \mathrm{d} \mu_{t} .
  \nonumber
\end{eqnarray}
The Codazzi equation implies $\nabla_{i} \mathring{h}^{\alpha}_{i j} =
\frac{n-1}{n}   \nabla_{j} H^{\alpha}$. From this formula and
(\ref{aopaodfsdh2}), we get
\begin{eqnarray}
  & &\ \ \ - \int_{M_{t}} \frac{f_{\sigma}^{p-1} \big\langle \mathring{h} , \nabla^{2}
  H \big\rangle}{\mathring{\gamma}^{1- \sigma}}   \mathrm{d} \mu_{t} \nonumber\\
    &  &=  \int_{M_{t}} \nabla_{i} \left( \frac{f_{\sigma}^{p-1}}{\mathring{\gamma}^{1-
  \sigma}} \mathring{h}_{i j}^{\alpha} \right) \nabla_{j} H^{\alpha}   \mathrm{d}
  \mu_{t} \label{eC2f2}\\
  &  &= \int_{M_{t}} \left[ \frac{( p-1 )
  f_{\sigma}^{p-2}}{\mathring{\gamma}^{1- \sigma}} \mathring{h}^{\alpha}_{i j}
  \nabla_{i} f_{\sigma}
   - \frac{( 1- \sigma ) f_{\sigma}^{p-1}}{\mathring{\gamma}^{2- \sigma}}
  \mathring{h}^{\alpha}_{i j}   \nabla_{i} \mathring{\gamma}  +
  \frac{f_{\sigma}^{p-1}}{\mathring{\gamma}^{1- \sigma}} \nabla_{i}
  \mathring{h}^{\alpha}_{i j} \right] \nabla_{j} H^{\alpha} \mathrm{d} \mu_{t} \nonumber\\
    & &\leq  \int_{M_{t}} \left[
    \frac{( p-1 ) f_{\sigma}^{p-2}}{\mathring{\gamma}^{1- \sigma}} | \mathring{h} | | \nabla
    f_{\sigma} |
     + \frac{f_{\sigma}^{p-1}}{\mathring{\gamma}^{2- \sigma}}
    | \mathring{h} | \left| \nabla \mathring{\gamma} \right|  +
    \frac{f_{\sigma}^{p-1}}{\mathring{\gamma}^{1- \sigma}} | \nabla H |
    \right]| \nabla H | \mathrm{d} \mu_{t} \nonumber\\
    &&\leq  \int_{M_{t}} \left[ \frac{( p-1 ) f_{\sigma}^{p-1}}{| \mathring{h} |}   | \nabla
  f_{\sigma} | |   \nabla H | + \frac{( 2C_{1} +1 )
  f_{\sigma}^{p}}{| \mathring{h} |^{2}}   |   \nabla H |^{2} \right] \mathrm{d} \mu_{t} . \nonumber
\end{eqnarray}

Putting (\ref{eC2f})-(\ref{eC2f2}) together, we get
\begin{equation}
  \int_{M_{t}} 2\left( 3 \sigma | h |^{2} -  \varepsilon\tilde{C}_{1} \right)
  f_{\sigma}^{p} \mathrm{d} \mu_{t} \leq \frac{C_{3} \sigma}{\varepsilon}
  \int_{M_{t}} \left[ \frac{p f_{\sigma}^{p-1}}{|\mathring{h}|}   | \nabla f_{\sigma}
  | |   \nabla H |  +  \frac{f_{\sigma}^{p}}{|\mathring{h}|^{2}}   |   \nabla H |^{2}
  \right]   \mathrm{d} \mu_{t} , \label{H2fsp}
\end{equation}
where  $C_{3}$ is a positive constant
depending on $n$.

Combining (\ref{eC2f1}), (\ref{H2fsp}) and Lemma \ref{ptf}, we obtain
\begin{eqnarray}
  \frac{\mathrm{d}}{\mathrm{d} t} \int_{M_{t}} f_{\sigma}^{p} \mathrm{d} \mu_{t} & = & p
  \int_{M_{t}} f_{\sigma}^{p-1} \frac{\partial f_{\sigma}}{\partial t} \mathrm{d}
  \mu_{t} - \int_{M_{t}} f_{\sigma}^{p} | H |^{2} \mathrm{d} \mu_{t} \nonumber\\
    & \leq & p \int_{M_{t}} \left[ f_{\sigma}^{p-1} \Delta f_{\sigma}  +
  \frac{4C_{1}  f_{\sigma}^{p-1}}{| \mathring{h} |} | \nabla f_{\sigma} | | \nabla H |
  \right. - \frac{2  f_{\sigma}^{p}}{5n | \mathring{h} |^{2}} | \nabla H |^{2}
  \nonumber\\
  &  & \left. + 2\left( 3 \sigma | h |^{2} -  \varepsilon\tilde{C}_{1} - \sigma n\right)
  f_{\sigma}^{p} \right] \mathrm{d} \mu_{t} \nonumber\\
  & \leq & p \int_{M_{t}} f_{\sigma}^{p-2} \left[ - ( p-1 ) | \nabla f_{\sigma}
  |^{2} + \left( 4C_{1} + \frac{C_{3}   \sigma  p}{\varepsilon} \right)
  \frac{f_{\sigma}}{|\mathring{h}|} | \nabla f_{\sigma} | | \nabla H | \right.
  \label{dtint}\\
  &  & \left. - \left( \frac{2}{5n} - \frac{C_{3}   \sigma}{\varepsilon}
  \right) \frac{f_{\sigma}^{2}}{|\mathring{h}|^{2}} | \nabla H |^{2} \right] \mathrm{d}
  \mu_{t} -2p \sigma n \int_{M_{t}} f_{\sigma}^{p} \mathrm{d} \mu_{t} . \nonumber
\end{eqnarray}

Now we show that the $L^{p}$-norm of $f_{\sigma}$ decays exponentially.

\begin{lemma}
  \label{pnorm}There exists a constant $C_{4}$ depending only on $M_{0}$ such
  that for all $p \geq 1/ \varepsilon$ and $\sigma \leq
  \varepsilon^{2} / \sqrt{p}$, we have
  \begin{equation*}  \left( \int_{M_{t}} f_{\sigma}^{p} \mathrm{d} \mu_{t} \right)^{\frac{1}{p}}
     <C_{4}   \mathrm{e}^{-2n \sigma t} . \end{equation*}
\end{lemma}

\begin{proof}
  The expression in the square bracket of the right hand side of (\ref{dtint})
  is a quadratic polynomial. With $p \geq 1/ \varepsilon$, $\sigma
  \leq \varepsilon^{2} / \sqrt{p}$ and $\varepsilon$ small enough, its
  discriminant satisfies $\big( 4C_{1} + \frac{C_{3}   \sigma
  p}{\varepsilon} \big)^{2} -4 ( p-1 ) \left( \frac{2}{5n} - \frac{C_{3}
  \sigma}{\varepsilon} \right) <0$. So, we obtain
  \begin{equation*}  \frac{\mathrm{d}}{\mathrm{d} t} \int_{M_{t}} f_{\sigma}^{p} \mathrm{d} \mu_{t}
     \leq -2 p \sigma n \int_{M_{t}} f_{\sigma}^{p} \mathrm{d} \mu_{t} . \end{equation*}
  This implies $\int_{M_{t}} f_{\sigma}^{p} \mathrm{d} \mu_{t} \leq
  \mathrm{e}^{-2 p \sigma n t} \int_{M_{0}} f_{\sigma}^{p} \mathrm{d} \mu_{0}$.
\end{proof}

Let $g_{\sigma} =f_{\sigma}   \mathrm{e}^{2 \sigma  t}$. By the
Sobolev inequality on submanifolds {\cite{MR0365424}} and a
Stampacchia iteration procedure, we obtain that $g_{\sigma}$ is
uniformly bounded for all $t$ (see {\cite{MR772132,lei2014sharp}}
for details). We then complete the proof of Theorem \ref{sa0h2}.

\section{A gradient estimate}

In the following, we derive an estimate for $| \nabla H |^{2}$ along the mean
curvature flow.

\begin{theorem}
  \label{dH2}If $M_{0}$ satisfies $|\mathring{h}|^{2} <
  \mathring{\gamma}_{\varepsilon}$, then for all $\eta \in ( 0, \varepsilon )$, there exists a number
  $\Psi ( \eta )$ depending on $\eta$ and $M_{0}$, such that
  \begin{equation*}  | \nabla H |^{2} < [ ( \eta | H | )^{4} + (\Psi ( \eta ))^{2} ] \mathrm{e}^{-
     \sigma t} . \end{equation*}
\end{theorem}

First, we derive an estimate for the time derivative of $| \nabla H |^{2}$.

\begin{lemma}
  \label{dtdH2}There exists a constant $B_{1}(>1)$ depending only on $n$, such
  that
  \begin{equation*}  \frac{\partial}{\partial t} | \nabla H |^{2} \leq \Delta | \nabla H |^{2} +B_{1} ( | H |^{2}
     +1 ) | \nabla h |^{2} . \end{equation*}
\end{lemma}

\begin{proof}
    We have the following evolution equation for $| \nabla H |^{2}$ (see \cite{baker2011mean}).
  \begin{equation*}  \frac{\partial}{\partial t} | \nabla H |^{2} = \Delta | \nabla H |^{2} -2  | \nabla^{2} H |^{2}
     +2 | \nabla H |^{2}  +h \ast h \ast \nabla h \ast \nabla h. \end{equation*}
  Here we use Hamilton's $\ast$ notation. For tensors $T$ and $S$, $T \ast S$
  means any linear combination of contractions of $T$ and $S$ with the metric.

  From the Cauchy-Schwarz inequality, we have $| h \ast h \ast \nabla h \ast
  \nabla h | \leq B_{0} | h |^{2} | \nabla h |^{2}$, where $B_{0}$ is a
  constant depending on $n$. From the pinching condition $| h |^{2} < \gamma$,
  we obtain $2 | \nabla H |^{2}  +B_{0} | h |^{2} | \nabla h |^{2} <B_{1} ( |
  H |^{2} +1 ) | \nabla h |^{2}$.
\end{proof}

Next we present the following estimates.

\begin{lemma}
  \label{threedtlap}Along the mean curvature flow, we have
  \begin{enumerateroman}
    \item $\frac{\partial}{\partial t} | H |^{4} \geq \Delta | H |^{4} -8 n  | H |^{2} | \nabla h
    |^{2} + \frac{4}{n} | H |^{6}$,

    \item $\frac{\partial}{\partial t} |\mathring{h}|^{2} \leq \Delta |\mathring{h}|^{2} - | \nabla h |^{2} +6n
    ( | H |^{2} +1 ) |\mathring{h}|^{2}$,

    \item $\frac{\partial}{\partial t} \left( | H |^{2} |\mathring{h}|^{2} \right) \leq \Delta \left( |
    H |^{2} |\mathring{h}|^{2} \right) - \frac{1}{2}   | H |^{2} | \nabla h |^{2}
    +B_{2} | \nabla h |^{2} +4n ( | H |^{2} +1 )^{2} |\mathring{h}|^{2}$, where
    $B_{2} >n$ is a positive constant.
  \end{enumerateroman}
\end{lemma}

\begin{proof}

  (i) By Lemma \ref{evo} (iii), we derive that
  \begin{equation*}  \frac{\partial}{\partial t} | H |^{4} = \Delta | H |^{4} -4 | H |^{2} | \nabla H |^{2} -2 \big|
     \nabla | H |^{2} \big|^{2} +4 | H |^{2} R_{2} +4 n | H |^{4} . \end{equation*}
  From Lemmas \ref{dA2} and \ref{R12}, we have $R_{2} \ge \frac{1}{n} | H
  |^{4}$ and $4 | H |^{2} | \nabla H |^{2} +2 \big| \nabla | H |^{2} \big|^{2}
  \leq 8 n | H |^{2} | \nabla h |^{2}$.

  (ii) The evolution equation of $| \mathring{h} |^{2}$ is
    \[ \frac{\partial}{\partial t} | \mathring{h} |^{2} = \Delta | \mathring{h} |^{2} -2 | \nabla \mathring{h} |^{2}
    +2R_{1} - \frac{2}{n} R_{2} -2 n  | \mathring{h} |^{2} . \]
    It follows from Lemma \ref{dA2} that $\frac{2}{n} | \nabla H |^{2} \leq |
  \nabla h |^{2}$. By Lemma \ref{R12} (i) and the pinching condition, we get
  $R_{1} - \frac{1}{n} R_{2} \leq |\mathring{h}|^{2} (
  3 |\mathring{h}|^{2} + \frac{| H |^{2}}{n} ) <3n ( | H |^{2} +1 )
  |\mathring{h}|^{2}$. Then we obtain the conclusion.

  (iii) It follows from the evolution equations that
  \begin{eqnarray*}
    \frac{\partial}{\partial t} \big( | H |^{2} |\mathring{h}|^{2} \big) & = & \Delta \big( | H |^{2}
    |\mathring{h}|^{2} \big) +2 | H |^{2} \Big( R_{1} - \frac{1}{n} R_{2} \Big)
    +2 |\mathring{h}|^{2} R_{2}\\
    &  & -2  | H |^{2} \Big( | \nabla h |^{2} - \frac{1}{n} | \nabla H |^{2}
    \Big) -2 |\mathring{h}|^{2} | \nabla H |^{2} -2 \big\langle \nabla | H |^{2} ,
    \nabla |\mathring{h}|^{2} \big\rangle .
  \end{eqnarray*}
 Here we have $\frac{2}{n} | \nabla H |^{2} \leq | \nabla h |^{2}$.
  Using Lemma \ref{R12} and the pinching condition, we have
  $$ 2 | H |^{2} \Big( R_{1} - \frac{1}{n} R_{2} \Big) +2 |\mathring{h}|^{2} R_{2}
  \leq 2  | H |^{2} |\mathring{h}|^{2} \left( 4 |\mathring{h}|^{2} + \frac{2}{n} | H
    |^{2} \right)    <4n ( | H |^{2} +1 )^{2} |\mathring{h}|^{2} .$$
   By Theorem \ref{sa0h2} and Young's inequality,
  there exists a positive constant $B_{2}$ such that $$\big| 2 \langle \nabla | H
  |^{2} , \nabla |\mathring{h}|^{2} \rangle\big| \leq 8  | H |   | \nabla H |
  |\mathring{h}|   | \nabla \mathring{h} | \leq \Big( B_{2} +
  \frac{1}{2}   | H |^{2} \Big) | \nabla h |^{2}.$$
  Thus we obtain the conclusion.
\end{proof}

\noindent\textit{Proof of Theorem \ref{dH2}.}
Define a function
\begin{equation*}  f= \left( | \nabla H |^{2} +5B_{1} B_{2}   |\mathring{h}|^{2} +4B_{1}   | H |^{2}
   |\mathring{h}|^{2} \right) \mathrm{e}^{\sigma t} - ( \eta   | H | )^{4} , \hspace{1em}
   \eta \in ( 0, \varepsilon ) . \end{equation*}
From Lemmas \ref{dtdH2} and \ref{threedtlap}, we obtain
\begin{eqnarray*}
    \Big( \frac{\partial}{\partial t} - \Delta \Big) f & \leq & \Big[ B_{1}   ( | H |^{2} +1 )
    | \nabla h |^{2} +5 B_{1} B_{2} \Big( - | \nabla h |^{2} +6n ( | H |^{2} +1
    ) |\mathring{h}|^{2} \Big) \Big.\\
    &  & +4B_{1} \Big( - \frac{1}{2}   | H |^{2} | \nabla h |^{2} +B_{2} |
    \nabla h |^{2} +4n ( | H |^{2} +1 )^{2} |\mathring{h}|^{2} \Big)\\
    &  & \Big. + \sigma \big( | \nabla H |^{2} +5 B_{1} B_{2}   |\mathring{h}|^{2}
    +4B_{1} | H |^{2} |\mathring{h}|^{2} \big) \Big] \mathrm{e}^{\sigma t}\\
    &  & - \eta^{4} \Big( -8 n  | H |^{2} | \nabla h |^{2} + \frac{4}{n} | H
    |^{6} \Big)\\
    & \leq & B_{3} ( | H |^{2} +1 )^{2} |\mathring{h}|^{2} \mathrm{e}^{\sigma t} - \frac{4
        \eta^{4}}{n} | H |^{6} ,
\end{eqnarray*}
where $B_{3}$ is a positive constant depending on $B_{1}$ and $B_{2}$.

By Theorem \ref{sa0h2}, we get
\begin{equation}
  \left( \frac{\partial}{\partial t} - \Delta \right) f \leq \left[ C_{0} B_{3} ( | H |^{2} +1
  )^{3- \sigma} - \frac{4 \eta^{4}}{n} | H |^{6} \right] \mathrm{e}^{- \sigma t} .
  \label{Paraf}
\end{equation}
By Young's inequality, \ the expression in the bracket of (\ref{Paraf}) is
bounded from above. Let $\Psi_{2} ( \eta )$ be its upper bound. Then we have
$\left( \frac{\partial}{\partial t} - \Delta \right) f \leq \Psi_{2} ( \eta ) \mathrm{e}^{- \sigma
t}$. It follows from the maximum principle that $f$ is bounded. This completes
the proof of Theorem \ref{dH2}.
\hfill\qedsymbol\\

\section{Convergence under sharp pinching condition}

We will follow Hamilton's idea in \cite{Hamilton} to use the
Myers theorem.

\begin{theorem}
  [\textbf{Myers}] Let $\Gamma$ be a geodesic of length $l$ on $M$. If the
  Ricci curvature satisfies $\operatorname{Ric}_M ( X ) \geq ( n-1 )
  \frac{\pi^{2}}{l^{2}}$, for each unit vector $X \in T_{x} M$, at any point
  $x \in \Gamma$, then $\Gamma$ has conjugate points.
\end{theorem}

We get the following lemma.

\begin{lemma}
  \label{ric}Suppose that $M$ is an $n$-dimensional $( n \geq 6 )$
  submanifold in $\mathbb{S}^{n+q}$ satisfying $|\mathring{h}|^{2} <
  \mathring{\gamma}_{\varepsilon}$ and $| \nabla H | <2  \eta^{2}   \max_{M} |
  H |^{2}$, where $0< \eta < \varepsilon$ and $\varepsilon$ is small enough.
  Then we have
  $ \frac{\min_{M} | H |^{2}}{\max_{M} | H |^{2}} >1- \eta $
  and $ \operatorname{diam}  M \leq ( 2  \eta  \max_{M} | H |
     )^{-1}$.
\end{lemma}

\begin{proof}
  By Proposition 2 of {\cite{MR1458750}}, the Ricci curvature of $M$ satisfies
  \begin{equation*}  \operatorname{Ric}_M \geq \frac{n-1}{n} \Big( n + \frac{2}{n} | H |^{2} - | h
     |^{2} - \frac{n-2}{\sqrt{n ( n-1 )}} | H | |\mathring{h}| \Big) . \end{equation*}
  From $| h |^{2} < \gamma - \varepsilon \omega$ and Lemma \ref{app} (iv), we
  obtain $$\operatorname{Ric}_M \geq \frac{n-1}{n}   \varepsilon   \omega >B_{4}
  \varepsilon | H |^{2},$$ where $B_{4}$ is a positive constant.

  Assume that $| H |$ attains its maximum at point $x\in M$. We consider
  all the geodesics of length $l= ( 4  \eta   \max_{M} | H | )^{-1}$ starting
from
  $x$. Since $\big| \nabla | H |^{2} \big| <4  \eta^{2} \cdot   \max_{M} | H |^{3}$, we have
  $| H |^{2} >(1- \eta) \max_{M} | H |^{2}$
  along such a geodesic. Thus we have $\operatorname{Ric}_M
  >B_{4} \varepsilon ( 1- \eta ) \max_{M}| H |^{2} > ( n-1 ) \pi^{2} /l^{2}$
  on each of
  these geodesics. By Myers' theorem, these geodesics can reach any
  point of $M$.

  Therefore, we obtain $\min_{M} | H |^{2} > ( 1- \eta ) \max_{M} | H |^{2}$
  and $\operatorname{diam}  M \leq 2l$.
\end{proof}

In the following we show that the mean curvature flow converges to a
point or a totally geodesic sphere under the assumption of Theorem
\ref{theo1}.

\begin{theorem}
  \label{Tfin}If $M_{0}$ satisfies $| h |^{2} < \gamma$ and $T$ is finite,
  then $F_{t}$ converges to a round point as $t \rightarrow T$.
\end{theorem}

\begin{proof}
  If $T$ is finite, we have $\max_{M_{t}} | h |^{2} \rightarrow \infty$ as $t
  \rightarrow T$. This can be shown by using an analogous argument as in the proof of the corresponding result
  in {\cite{MR2739807,liu2012mean}}. Let $| H |_{\min} =
  \min_{M_{t}} | H |$, $| H |_{\max} = \max_{M_{t}} | H |$. From the preserved
  pinching condition, we get $| H |_{\max} \rightarrow \infty$ as $t
  \rightarrow T$.

  By Theorem \ref{dH2}, for any $\eta \in ( 0, \varepsilon )$, we have $|
  \nabla H | < ( \eta   | H | )^{2} + \Psi ( \eta )$. Since $| H |_{\max}
  \rightarrow \infty$ as $t \rightarrow T$, there exists a time $\tau$
  depending on $\eta$, such that for $t> \tau$, there holds $| H |_{\max}^{2} > \Psi (
  \eta ) / \eta^{2}$. Then we have $| \nabla H | <2  \eta^{2} | H
  |^{2}_{\max}$. Using Lemma \ref{ric}, we obtain $\operatorname{diam}  M_{t}
  \rightarrow 0$ and $| H |_{\min} / | H |_{\max} \rightarrow 1$ as $t
  \rightarrow T$.

  To prove the flow converges to a round point, we magnify the metric of the
  ambient space such that the submanifold maintains its volume along the flow.
  Using the same argument as in {\cite{liu2012mean}}, we can prove that the
  rescaled mean curvature flow converges to a totally umbilical sphere as the
  reparameterized time tends to infinity.
\end{proof}

\begin{theorem}
  \label{Tinf}If $M_{0}$ satisfies $| h |^{2} < \gamma$ and $T= \infty$, then
  $F_{t}$ converges to a totally geodesic sphere as $t \rightarrow \infty$.
\end{theorem}

\begin{proof}
  With the assumption $T= \infty$, we will show that $| H |$ decays
  exponentially.

  Suppose $| H |_{\max}^{2} \cdot \mathrm{e}^{\sigma t/2}$ is unbounded. Then for
  a small positive number $\eta$, there exists a time $\theta$, such that $| H
  |_{\max}^{2} ( \theta ) \cdot \mathrm{e}^{\sigma \theta /2} > \Psi ( \eta ) /
  \eta^{2}$. By Theorem \ref{dH2}, we have $| \nabla H | <2 \eta^{2}   | H
  |^{2}_{\max}$ on $M_{\theta}$. By Lemma \ref{ric}, we have $| H |^{2}_{\min}
  ( \theta ) > ( 1- \eta ) | H |^{2}_{\max} ( \theta ) > \frac{1-
  \eta}{\eta^{2}}   \Psi ( \eta )   \mathrm{e}^{- \sigma \theta /2}$. This
  together with Theorem \ref{dH2} yields $| \nabla H |^{2} < ( \eta | H |
  )^{4} + \frac{\eta^{4}}{( 1- \eta )^{2}} | H |_{\min}^{4} ( \theta ) \cdot
  \mathrm{e}^{\sigma ( \theta -t )}$. From the evolution equation of $| H |^{2}$,
  we have
  \begin{equation}
    \frac{\partial}{\partial t} | H |^{2} > \Delta | H |^{2} + \frac{1}{n} | H |^{4} - \frac{1}{2n} |
    H |_{\min}^{4} ( \theta ) \cdot \mathrm{e}^{\sigma ( \theta -t )} .
    \label{dtH2gr}
  \end{equation}
  Using the maximum principle, we get $| H |^{2} \geq | H |^{2}_{\min} (
  \theta )$ if $t \geq \theta$. Then (\ref{dtH2gr}) yields $\frac{\partial}{\partial t} | H |^{2}
  > \Delta | H |^{2} + \frac{1}{2n} | H |^{4}$ for $t \geq \theta$.
  Hence, $| H |^{2}$ will tend to infinity in finite time. This contradicts
  the infinity of $T$. Therefore, we obtain $| H |^{2} <C \mathrm{e}^{- \sigma
  t/2}$.

  From Theorem \ref{sa0h2}, we have $| h |^{2} = |\mathring{h}|^{2} + \frac{1}{n} |
  H |^{2} \leq C \mathrm{e}^{- \sigma  t/2}$. Since $| h | \rightarrow 0$ as
  $t \rightarrow \infty$, $M_{t}$ converges to a totally geodesic submanifold
  as $t \rightarrow \infty$.
\end{proof}

\section{Convergence under weakly pinching condition}

Assume that $M_{0}$ is a closed submanifold in $\mathbb{S}^{n+q}$
whose squared norm of the traceless second fundamental form
satisfies $|\mathring{h}|^{2} \leq \mathring{\gamma}$.

\noindent\textit{Proof of Theorem \ref{theo2}.}
  Recalling the proof of Theorem
  \ref{pinch}, we have
  \begin{eqnarray}
    &  & \Big( \frac{\partial}{\partial t} - \Delta \Big) \left( |\mathring{h}|^{2} - \mathring{\gamma}
    \right) \nonumber\\
    & \leq & \Big[ - \frac{2 ( n-1 )}{n ( n+2 )} + \mathring{\gamma}'
    +2 | H |^{2} \mathring{\gamma}''   \Big] | \nabla H |^{2}
    \label{weakpinch}\\
    &  & +2 \left( |\mathring{h}|^{2} - \mathring{\gamma} \right) \Big( 2
    \mathring{\gamma} + \frac{1}{n} | H |^{2} -n- | H |^{2}
    \mathring{\gamma}' +2P_{2} \Big) +2 \left( |\mathring{h}|^{2} -
    \mathring{\gamma} \right)^{2} . \nonumber
  \end{eqnarray}
  The coefficient of $| \nabla H |^{2}$ is negative. Using the strong maximum
  principle, we have either $|\mathring{h}|^{2} < \mathring{\gamma}$ at some $t_{0}
  \in [ 0,T )$, or $|\mathring{h}|^{2} \equiv \mathring{\gamma}$ for all $t \in [ 0,T
  )$.

  If $|\mathring{h}|^{2} < \mathring{\gamma}$ at some $t_{0} \geq 0$, it reduces to the
  case of Theorem \ref{theo1}.

  If $|\mathring{h}|^{2} \equiv \mathring{\gamma}$ for $t \in [ 0,T )$, we have
  $\nabla H \equiv 0$ for all $t$. Since $\mathring{\gamma} >0$ and
  $\mathring{\gamma} ( 0 ) < n$, $M_{t}$ is neither a totally
  umbilical sphere nor a Clifford minimal hypersurface in $\mathbb{S}^{n+1}$. By Lemma \ref{gminab} (i), (iii) and
    Theorem 3 of \cite{MR1241055}, we
  obtain that $| H | \neq 0$ and $M_{t}$ is an isoparametric hypersurface in an
  $( n+1 )$-dimensional totally geodesic sphere
  \begin{equation*}  \mathbb{S}^{n-1} ( r_{1} ) \times \mathbb{S}^{1} ( r_{2} ) \subset
     \mathbb{S}^{n+1} \subset \mathbb{S}^{n+q} , \end{equation*}
  where $r_{1} = \frac{1}{\sqrt{1+ \lambda^{2}}}$, $r_{2} =
  \frac{\lambda}{\sqrt{1+ \lambda^{2}}}$ and $\lambda = \frac{| H | + \sqrt{|
  H |^{2} +4 ( n-1 )}}{2 ( n-1 )} > \sqrt{\frac{1}{n-1}}$.

  Notice that $\lambda$ is the ($n-1$)-multiple principal curvature and $-
  \frac{1}{\lambda}$ is the other one. Thus we have $| H | = ( n-1 ) \lambda -
  \frac{1}{\lambda}$ and $| h |^{2} = ( n-1 ) \lambda^{2} +
  \frac{1}{\lambda^{2}}$. Since $M_{t}$ is a family of hypersurfaces in
  $\mathbb{S}^{n+1}$,  the evolution equation of $|H|$ becomes $\left( \frac{\partial}{\partial t} - \Delta
  \right) | H | = | H | ( | h |^{2} +n )$. Hence,
  \[ \frac{\mathrm{d}}{\mathrm{d} t} \Big( ( n-1 ) \lambda - \frac{1}{\lambda}
  \Big) = \Big( ( n-1 ) \lambda - \frac{1}{\lambda} \Big) \Big( ( n-1
  ) \lambda^{2} + \frac{1}{\lambda^{2}} +n \Big). \]
  This is equivalent to
  $ \frac{\mathrm{d}}{\mathrm{d} t} r_{1}^{2} =2-2n+2n r_{1}^{2} $.
  Solving this ODE, we obtain
  \begin{equation*}  r_{1}^{2} = \frac{n-1}{n} ( 1-d \cdot \mathrm{e}^{2n t} ) , \end{equation*}
  where $d \in ( 0,1 )$ is a constant of integration.

  From the solution of $r_{1}$, we see that the maximal existence time $T=-
  \frac{\log  d}{2n}$. Hence $M_t$ converges to a great circle as $t \rightarrow T$.
  This completes the proof of Theorem \ref{theo2}.
\hfill\qedsymbol\\

\end{document}